\documentclass[11pt,a4wide]{amsart}
\usepackage[english]{babel}
\usepackage[T1]{fontenc}

\usepackage{mathtools} 
\usepackage{amsmath,amssymb,amsthm}
\usepackage{subcaption}
\usepackage[shortlabels]{enumitem}
\usepackage{hyperref}

\newtheorem{theorem}{Theorem}
\newtheorem{lemma}[theorem]{Lemma}

\newtheorem{definition}[theorem]{Definition}
\newtheorem{corollary}[theorem]{Corollary}
\newtheorem{remark}[theorem]{Remark}
\numberwithin{equation}{section}
\newcommand{\vertiii}[1]{{\left\vert\kern-0.25ex\left\vert\kern-0.25ex\left\vert #1 
		\right\vert\kern-0.25ex\right\vert\kern-0.25ex\right\vert}}

\newcommand{\isdef}{\mathrel{\mathrel{\mathop:}=}}
\newcommand{\supp}{\operatorname{supp}}

\newcommand{\on}[1]{\operatorname{#1}}
\newcommand{\drm}{\mathrm{d}}

\newcommand{\Vbfm}{\mathbf{V}}

\newcommand{\nbfm}{\mathbf{n}}

\newcommand{\xbfm}{\mathbf{x}}
\newcommand{\ybfm}{\mathbf{y}}

\newcommand{\Rbbb}{\mathbb{R}}

\newcommand{\Acal}{\mathcal{A}}

\newcommand{\Ccal}{\mathcal{C}}
\newcommand{\Dcal}{\mathcal{D}}
\newcommand{\Ecal}{\mathcal{E}}

\newcommand{\Hcal}{\mathcal{H}}

\newcommand{\Kcal}{\mathcal{K}}
\newcommand{\Lcal}{\mathcal{L}}
\newcommand{\Mcal}{\mathcal{M}}

\newcommand{\Ocal}{\mathcal{O}}

\newcommand{\Rcal}{\mathcal{R}}
\newcommand{\Scal}{\mathcal{S}}
\newcommand{\Tcal}{\mathcal{T}}

\newcommand{\Vcal}{\mathcal{V}}
\newcommand{\Wcal}{\mathcal{W}}

\newcommand{\alphabfm}{{\boldsymbol{\alpha}}}

\newcommand{\kappabfm}{{\boldsymbol{\kappa}}}

\newcommand{\nubfm}{{\boldsymbol{\nu}}}

\newcommand{\D}{\operatorname{D}\!}

\begin{document}
\title[BIEs for the heat equation in time-dependent domains]
{Boundary integral operators for the heat equation in time-dependent domains}
\author{Rahel Br\"ugger}
\author{Helmut Harbrecht}
\address{Rahel Br\"ugger and Helmut Harbrecht, 
Departement Mathematik und Informatik, 
University of Basel, 
Spiegelgasse 1, 4051 Basel, Schweiz.}
\email{\{ra.bruegger,helmut.harbrecht\}@unibas.ch}
\author{Johannes Tausch}
\address{Department of Mathematics, 
Southern Methodist University, Dallas, TX 75275, USA.}
\email{tausch@smu.edu}

\begin{abstract}
  This article provides a functional analytical framework
  for boundary integral equations of the heat equation
  in time-dependent domains. More specifically, we consider a
  non-cylindrical domain in space-time that is the
    $C^2$-diffeomorphic image of a cylinder, i.e., the tensor product
    of a time interval and a fixed domain in space. On the
  non-cylindrical domain, we introduce Sobolev spaces, trace lemmata
  and provide the mapping properties of the layer operators by
  mimicking the proofs of Costabel \cite{Costabel1990}. Here it
    is critical that the Neumann trace requires a correction term
  for the normal velocity of the moving boundary. Therefore,
  one has to analyze the situation carefully.
\end{abstract}


\maketitle

\section{Introduction}
Boundary integral equations are a well-known technique to solve elliptic 
partial differential equations, see for example \cite{Sauter2010, Steinbach2008}. 
For parabolic equations on \emph{time-independent}, so-called \emph{cylindrical}
domains, Sobolev spaces and the mapping properties of the layer operators for 
the heat equation are introduced in
\cite{Costabel1990,noon88}. To the best of our knowledge, 
no such theory exists for \emph{time-dependent} or so-called \emph{non-cylindrical} 
domains, or simply, \emph{tube}. Therefore, the aim of this article is to 
extend the theory from cylindrical domains to non-cylindrical domains.

To that end, we consider a special class of non-cylindrical domains. We
fix a cylindrical domain which serves as a reference domain and define
the non-cylindrical domain as the image under a time-dependent $C^2$-diffeomorphism. 
In this setting, different approaches are possible to establish analogue 
integral equations and properties to integral operators
on a cylindrical domain. 

A first approach could be to exploit the fact that the fundamental solution does
not use boundary data and is thus defined on the free space $\Rbbb \times 
\Rbbb^d$. Therefore, it is the same for a cylindrical and a non-cylindrical 
domain and allows to state the integral operators in cylindrical and non-cylindrical 
domains. To establish the mapping properties of the integral operators, one could 
make use of the equivalence of norms on the tube and on the cylindrical domain 
by establishing equivalence results of the fundamental solution, evaluated on the 
tube and on the cylindrical domain. The problem is that the Neumann trace, which 
will be considered here, contains an additional term involving the normal velocity of 
the tube. Therefore, one has to come up with a solution how to deal with this.

A second approach could be to map back the partial differential equation 
from the non-cylindrical domain onto the cylindrical domain. The advantage is 
that one now considers a cylindrical domain, for which more theory is available. 
The drawback is that the differential equation in the reference
domain is more complicated because of
time and space dependent coefficients. Finding a fundamental solution 
is more difficult and one could for example pursue the parametrix ansatz,
taken in \cite{Friedman1983}, and then find the according mapping properties 
of the respective layer operators.

A third approach considers the partial differential equation on the
non-cylindrical domain. Here, the partial differential equation
is simple, but the domain is more involved in contrast to the second
approach. This approach was used in \cite{Tausch2019}, but without
the corresponding Sobolev spaces and mapping properties of the
integral operators. Since we already did computations in
\cite{Bruegger2020} based on \cite{Tausch2019}, and since
\cite{Costabel1990} provides a self-contained analysis of the mapping
properties of the layer operators for the heat equation in a
cylindrical domain, we choose this approach for this article.

Although we follow the argumentation line of Costabel \cite{Costabel1990}, 
we repeat the proofs here in the non-cylindrical setting for the reader's convenience, 
since we have to use the appropriate function spaces and the correct Neumann traces. 
We would like to emphasize that, once one has the appropriate Neumann trace operator 
at hand and the mapping properties of the trace operators, the ideas of \cite{Costabel1990} 
can be followed directly. We indicate the needed adaptations in the article.

The remainder of this article is organized as follows:
In Section \ref{sec.anisotropic_sobolev_spaces}, we introduce anisotropic Sobolev 
spaces on cylindrical and non-cylindrical domain, which are used for the 
mapping properties. Section \ref{sec.Dirichlet_traces} is dedicated to the Dirichlet 
traces and the existence and uniqueness of solutions of the Dirichlet problem. 
In Section \ref{sec.neumann_trace_op}, we have a look at the appropriate Neumann 
trace. The main result is presented in Section \ref{sec.calderon_op}, where we establish 
the mapping properties of the integral operators and the existence and 
uniqueness of solutions of the Neumann problem. In Section \ref{sec.conclusion}, 
we state some concluding remarks. 

\section{Anisotropic Sobolev spaces}
\label{sec.anisotropic_sobolev_spaces}
In order to study the heat equation, we shall introduce appropriate 
aniso\-tro\-pic Sobolev spaces on cylindrical domains. From these spaces, 
we will then derive Sobolev spaces on time-dependent domains.
\subsection{Anisotropic Sobolev spaces on cylindrical domains}
Let $\Omega_0 \subset \Rbbb^d$, $d \geq 2$, be a Lipschitz domain in the spatial 
variable with boundary $\Gamma_0 \isdef \partial \Omega_0$ and let $0<T<\infty$. Then, 
the product set $Q_0 \isdef (0,T) \times \Omega_0 \subset \Rbbb^{d+1}$ forms a 
time-space cylinder with the lateral boundary $\Sigma_0\isdef (0,T)\times \Gamma_0$. 
The appropriate function spaces for parabolic problems in time invariant domains, 
i.e.~in cylindrical domains, are the anisotropic Sobolev spaces defined by 
\[ 
H^{r,s} (Q_0) \isdef L^2\big((0,T); H^r (\Omega_0)\big) \cap H^s\big((0,T); L^2(\Omega_0)\big)
\]
for $r,s \in \mathbb{R}_{\geq 0}$, see, e.g., \cite{Chapko1998, Costabel1990,Lions1972a}.
The corresponding boundary spaces are
\[ 
H^{r,s} (\Sigma_0) \isdef L^2\big((0,T); H^r (\Gamma_0)\big) \cap H^s\big((0,T); L^2(\Gamma_0)\big)
\]
Note that these spaces are well-defined for $r \leq 1$ (while $s\ge 0$ is arbitrary) if 
$\Gamma_0$ is Lipschitz.

\begin{remark}
The space $H^{r,s} (Q_0)$ consists of all functions $u \in 
L^2 (Q_0)$, where the $L^2 (Q_0)$-norm of the partial derivatives 
$\partial_{\xbfm}^{\alphabfm} \partial_t^\beta u (t, \xbfm)$ is finite 
for all $|\alphabfm| \leq \lambda r$, $\beta\leq(1-\lambda) s$, 
and $\lambda \in [0,1]$. 
\end{remark}

With these definitions at hand, we can moreover define spaces for functions 
with zero initial condition by setting
\[
H^{r,s}_{;0,}(Q_0)  \isdef L^2 \big( (0,T); H^r (\Omega_0)\big) \cap H_{0,}^{s} \big( (0,T); L^2 (\Omega_0) \big) ,
\]
where
\[ H_{0,}^s \big( (0,T); L^2 (\Omega_0) \big) \isdef \Big\lbrace u = \tilde{u}|_{(0,T)} \colon \tilde{u} \in H^s \big( (-\infty,T); L^2 (\Omega_0) \big)\colon \tilde{u} (t) = 0 \text{ for } t<0 \Big\rbrace.\]
Note that we adopted the notation from \cite{Dohr2019a,
  Dohr2019}. In addition,
we can define functions which vanish at $t=T$ by setting
\begin{equation*}
H^{r,s}_{;,0}(Q_0)  \isdef L^2 \big( (0,T); H^r (\Omega_0)\big) \cap H_{,0}^{s} \big( (0,T); L^2 (\Omega_0) \big) ,
\end{equation*}
where in complete analogy
\[ H_{,0}^s \big( (0,T); L^2 (\Omega_0) \big) \isdef \Big\lbrace u = \tilde{u}|_{(0,T)} \colon \tilde{u} \in H^s \big( (0,\infty); L^2 (\Omega_0) \big)\colon \tilde{u} (t) = 0 \text{ for } t>T \Big\rbrace.\]

As in the elliptic case, we can also include (spatial) zero boundary 
conditions into the function spaces by setting
\[
\begin{aligned}
H^{r,s}_{0;0,} (Q_0) \isdef L^2 \big( (0,T); H^r_0 (\Omega_0 ) \big) \cap H_{0,}^s \big( (0,T); L^2(\Omega_0) \big), \\
H^{r,s}_{0;,0} (Q_0) \isdef L^2 \big( (0,T); H^r_0 (\Omega_0 ) \big) \cap H_{,0}^s \big( (0,T); L^2(\Omega_0) \big),
\end{aligned}
\]
where the spaces include zero initial and end conditions, 
respectively. On the boundary, we introduce
\[
\begin{aligned}
H^{r,s}_{;0,} (\Sigma_0)  &\isdef L^2  \big( (0,T); H^r (\Gamma_0) \big) \cap H^{s}_{0,} \big( (0,T); L^2(\Gamma_0) \big), \\
H^{r,s}_{;,0} (\Sigma_0)  & \isdef L^2 \big( (0,T); H^r (\Gamma_0) \big) \cap H^{s}_{,0} \big( (0,T); L^2(\Gamma_0) \big).
\end{aligned}
\]
These spaces are the closures of $H^{r,s}(\Sigma_0)$ for zero start and end 
condition, respectively, compare \cite[Section 2.3]{Dohr2019a}.

By duality we have
\[ 
H^{-r, -s}_{;0,} (Q_0) = \big[H_{0;,0}^{r,s} (Q_0)\big]' \quad \text{for } r - \frac{1}{2} \notin \mathbb{Z}
\]
according to \cite{Costabel1990}.
The anisotropic Sobolev spaces on the boundary with 
negative smoothness index are defined by
\begin{align*}
H^{-r,-s}_{;,0} (\Sigma_0) &\isdef \big[ H^{r,s}_{;0,} (\Sigma_0) \big]',\\
H^{-r,-s}_{;0,} (\Sigma_0) &\isdef \big[ H^{r,s}_{;,0} (\Sigma_0) \big]',\\
\widetilde{H}^{-r,-s} (\Sigma_0) &\isdef \big[ H^{r,s} (\Sigma_0) \big]',
\end{align*}
see \cite[Section 2.3]{Dohr2019a}. Moreover, according to \cite[Remark 2.1]{Dohr2019a}, 
for $r \geq 0$ and $0 \leq s < \frac{1}{2}$ it holds $H^{r,s}(\Sigma_0) = H^{r,s}_{;0,}(\Sigma_0) 
= H^{r,s}_{;,0}(\Sigma_0)$ and, therefore, the above introduced dual spaces are equal
and we simply write $H^{-r,-s}(\Sigma_0)$.

\begin{remark}
We would like to clarify the intuition behind the slightly cumbersome 
notation. In $H^{r,s}_{0;,} (Q_0)$, a zero before the semicolon indicates 
a zero boundary condition in space. After the semicolon, a zero initial 
condition can be indicated by writing a zero between the semicolon and 
the comma. Whereas, a present zero after the comma stands for a zero 
end condition. Thus, this notation allows to see the spatial and temporal 
boundary condition at one glance.
\end{remark}

\subsection{Anisotropic Sobolev spaces on non-cylindrical domains}
\label{subsec.anisotropic_Sobolev}
Having at hand the Sobolev spaces defined on cylindrical domains, 
we can also introduce Sobolev spaces on non-cylindrical domains. 
Non-cylindrical domains consist of a spatial domain, which we denote 
by $\Omega_t$. The subscript $t$ indicates that the spatial domain 
might differ for every point of time. To obtain a non-cylindrical domain 
$Q_T$ we set
\[ 
Q_T \isdef \bigcup_{0< t< T} \big( \lbrace t \rbrace \times \Omega_t \big). 
\]
This domain has a lateral boundary $\Sigma_T$ defined by
\[ 
\Sigma_T \isdef \underset{0 < t < T}{\bigcup} \big( \{t \} \times\Gamma_t \big),
\]
where $\Gamma_t\isdef \partial \Omega_t$.

For every point of time $t$, we assume to have a smooth diffeomorphism 
$\boldsymbol\kappa$, which maps the initial domain $\Omega_0$ onto 
the time-dependent domain $\Omega_t$. In accordance with 
\cite{Moubachir2006}, we write
\begin{equation}\label{eq:deformation}
\boldsymbol\kappa\colon [0, T ] \times \mathbb{R}^d \to \mathbb{R}^d, 
\quad (t,\xbfm) \mapsto \boldsymbol\kappa(t, \xbfm)
\end{equation}
to emphasize the dependence of the mapping $\boldsymbol\kappa$ on 
the time, where we have $\boldsymbol\kappa(t,\Omega_0) = \Omega_t$. 
Especially, $\Omega_t$ is also a Lipschitz domain for all $t\in [0,T]$.

The domains $\Omega_t$ each have a spatial normal $\nbfm_t$, 
which we will also denote by $\nbfm$ if it is clear from the context. 
Besides having a spatial normal, we also have a time-space normal 
$\nubfm$. We can write the time-space normal as
\begin{equation}
\label{eq.space_time_normal}
\nubfm = \frac{1}{\sqrt{1+v_{\nubfm}^2}} \begin{bmatrix}
v_{\nubfm} \\ \nbfm
\end{bmatrix}
\end{equation}
for some appropriate $v_{\nubfm} \in \Rbbb$. According to \cite{Dziri2001}, it holds
\[  v_{\nubfm}= - \langle \Vbfm, \nbfm \rangle \]
for the vector field $\Vbfm$, which deforms the cylinder $Q_0$ into the tube 
$Q_T$ and for which the relation $\Vbfm = \partial_t \kappabfm \circ \kappabfm^{-1}$ holds.

We introduce the 
non-cylindrical analogues of the Sobolev spaces by setting
\[H^{r,s}(Q_T) \isdef \big\lbrace v \in L^2(Q_T)\colon v \circ \kappabfm \in H^{r,s} (Q_0) \big\rbrace \]
where the composition with $ \boldsymbol\kappa$ 
only acts on the spatial component. Due to the chain rule, $v\circ\boldsymbol\kappa$ 
and $v$ have the same Sobolev regularity, provided that the mapping $\boldsymbol\kappa$ 
is smooth enough, see for example \cite[Theorem 3.23]{McLean2000} for the elliptic case. 
For what follows, we assume that that $\kappabfm \in C^2\big([0,T]\times \mathbb{R}^d\big)$ 
satisfies
\begin{equation}\label{eq.uniformity_cond}
\|\boldsymbol\kappa(t, \xbfm) \|_{C^2([0,T]\times \mathbb{R}^d ; \mathbb{R}^d)}, 
\|\boldsymbol\kappa(t, \xbfm)^{-1} \|_{C^2([0,T] \times \mathbb{R}^d ; \mathbb{R}^d)} \leq C_{\kappabfm}
\end{equation}
for some constant $C_{\kappabfm} \in (0, \infty)$ as in \cite[pg.~826]{Harbrecht2016}.
We define the norm of $H^{r,s}(Q_T)$ as 
\[
	 \| u\|_{H^{r,s}(Q_T)} = \| u\circ \kappa \|_{H^{r,s}(Q_0)}
\]
for $r,s \geq 0$.
Notice that the Sobolev spaces on the boundary are defined 
in a similar manner.

\begin{remark}
(i)
The space $H^{r,s}(Q_T)$ contains all functions such that
$u \circ \kappabfm \in H^{r,s}(Q_0)$. This means that
$\partial_{\xbfm}^{\alphabfm} \partial_t^\beta (u \circ \kappabfm)
\in L^2 (Q_0)$ for all $|\alphabfm| \leq \lambda r$, $\beta\leq(1-\lambda) s$, 
and $\lambda \in [0,1]$. According to \eqref{eq.uniformity_cond}, the 
partial derivatives $\partial^{\alphabfm}_\xbfm\partial_t^\beta\kappabfm$ 
exist and are uniformly bounded for all $|\alphabfm|+\beta\leq 2$. 

(ii)
Consider a function $u\in L^2(Q_T)$ with partial derivatives 
$\partial_{\bf x}^{\alphabfm}\partial_t^\beta u\in L^2(Q_T)$ for 
all $|\alphabfm| \leq \lambda r$, $\beta\leq(1-\lambda) s$, and 
$\lambda \in [0,1]$. When computing the time-derivative of 
$u\circ\kappabfm$, we obtain also a spatial derivative as the 
following shows
\begin{equation}\label{eq.id_stefan2}
  (\partial_t u)\circ \kappabfm = \partial_t(u \circ \kappabfm) 
	-\big\langle(\on{D}\kappabfm)^{-\intercal}
	\nabla(u \circ \kappabfm) , \partial_t \kappabfm\big\rangle.
\end{equation}
Hence, it holds $u\in H^{r,s}(Q_T)$ only if $r\ge s$ since 
the temporal derivative $\partial_t^\beta (u\circ\kappa)$ involves also 
spatial partial derivatives $\partial_{\bf x}^{\alphabfm} u$ up to the order
$|\alphabfm|=\beta$ besides the temporal derivative $\partial_t^\beta u$.

(iii)
Due to the uniformity condition \eqref{eq.uniformity_cond}, 
we have as in \cite{Harbrecht2016}
\[ 
0 < \underline{\sigma} \leq \min \lbrace \sigma (\D\boldsymbol{\kappa}) \rbrace 
\leq \max \lbrace \sigma (\D\boldsymbol{\kappa}) \rbrace \leq \overline{\sigma} < \infty,
\]
where $\D\boldsymbol{\kappa}$ denotes the Jacobian of $\boldsymbol{\kappa}$
and $\sigma(\D\boldsymbol{\kappa})$ denotes its singular values. Especially, 
as in \cite[Remark 1, pg.~827]{Harbrecht2016}, we may assume 
$\det (\D\boldsymbol{\kappa})$ to be positive. 

We can define the dual of $H^{r,s}_{0;,0}(Q_T)$ in two different ways, namely
\[
   \| u \|_{H^{-r,-s}_{;0,} (Q_T)} = \sup_{\widetilde{v} \in H^{r,s}_{0;,0} (Q_0)} \int_{Q_0} 
   	(u \circ \kappabfm) \widetilde{v} \, \drm (\xbfm, t)
\]
and
\[ 
  \vertiii{u}_{H^{-r, -s}_{;0,} (Q_T)} = \sup_{v \in H^{r,s}_{0;,0}(Q_T)} \int_{Q_T} u v \,\drm (\xbfm, t).
\]
We show that these norms are equivalent. On one hand, there holds
\begin{align*}
\vertiii{u}_{H^{-r, -s}_{;0,} (Q_T)} 
&= \sup_{v \in H^{r,s}_{0;,0} (Q_T)} \frac{\int_{Q_0} (u \circ \kappabfm) (v \circ \kappabfm) \det ( \D \kappabfm) \, \drm (\xbfm, t)}
{\| v \|_{H^{r,s}_{0;,0} (Q_T)}} \\
& \leq \| u \|_{H^{-r, -s}_{;0,} (Q_T)} \sup_{v \in H^{r,s}_{0;,0}(Q_T)} \frac{ \big\| (v \circ \kappabfm) 
\det(\D \kappabfm) \big\|_{H^{r,s}_{0;,0} (Q_0)}}{\| v \circ \kappabfm \|_{H^{r,s}_{0;,0} (Q_0)}}\\
& \lesssim \| u \|_{H^{-r, -s}_{;0,} (Q_T)},
\end{align*}
where we used the definition of the norm on $H^{r,s}_{0;,0} (Q_T)$ 
for $s, r \geq 0$ and that the pointwise multiplication with a 
smooth function is a continuous operation.
On the other hand, we likewise find
\begin{align*}
\| u \|_{H^{-r,-s}_{;0,} (Q_T)} 
&= \sup_{\widetilde{v} \in H^{r,s}_{0;,0}(Q_0)} \frac{\int_{Q_T} u (\widetilde{v} \circ \kappabfm^{-1})
\det (\D \kappabfm^{-1}) \, \drm (\xbfm, t)}{\| \widetilde{v} \|_{H^{r,s}_{0;,0} (Q_0)} }\\
& \leq \vertiii{u}_{H^{-r,-s}_{;0,} (Q_T)} \sup_{\widetilde{v} \in H^{r,s}_{0;,0} (Q_0)} 
\frac{\big\|(\widetilde{v} \circ \kappabfm^{-1})\det (\D \kappabfm ^{-1}) \big\|_{H^{r,s}_{0;,0}(Q_T)}}
{\| \widetilde{v} \|_{H^{r,s}_{0;,0} (Q_0)}} \\
& \lesssim \vertiii{u}_{H^{-r, -s}_{;0,} (Q_T)},
\end{align*}
Hence, both duality pairings result in the same dual spaces and 
we can say that $H_{0;,0}^{r,s} (Q_T)$ and $H^{-r, -s}_{;0,} (Q_T)$ 
are indeed dual, as likewise for the other pairings.
\end{remark}

Finally, let the space $\mathcal{V}(Q_T)$ consist of all the functions $v$ 
with $v\circ\kappabfm\in\mathcal{V} (Q_0)$ and
\[ 
\mathcal{V}(Q_0)\isdef\big\{u \in L^2\big((0,T); H^1(\Omega_0)\big) \colon
\partial_t u \in L^2\big((0,T); H^{-1} (\Omega_0)\big)\big\}.
\]
The norm on this space is given by
\[ 
\| u \|_{\Vcal(Q_0)}^2 \isdef \| u \|_{H^{1,0} (Q_0)}^2 + \| \partial_t u \|_{H^{-1,0} (Q_0)}^2.
\]
Note that the space $\Vcal (Q_0)$ is a dense subspace of $H_{;,}^{1,\frac{1}{2}} (Q_0)$, 
which follows according to \cite[Formula (2.2)]{Costabel1990} from the interpolation result
\begin{equation}\label{eq.interpol_result}
 L^2(I; X) \cap H^1 (I;Y) \subset H^{\frac12} \big( I; [X,Y]_{\frac12} \big) \cap C \big( \bar{I}; [X,Y]_{\frac12} \big) 
\end{equation}
for $X \subset Y$ being Hilbert spaces.

\section{Dirichlet Problem}\label{sec.Dirichlet_traces}
\subsection{Dirichlet trace operator on cylindrical domains}
\label{subsec.trace_op_cylindrical}
We first introduce the notion of traces with respect to cylindrical domains. 
According to \cite[Section 2.3]{Dohr2019a}, we can define the (interior)
Dirichlet trace for a function $u \in C^1(\overline{Q}_0)$ as 
\[ 
\gamma_0 u (t, \xbfm) \isdef \lim_{\Omega_0 \ni \ybfm \to \xbfm \in \Gamma_0} 
u (t, \ybfm) \quad \text{for } (t, \xbfm) \in \Sigma_0. 
\]
We thus have $\gamma_0 u = u|_{\Sigma_0}$. We can introduce 
a similar operator on anisotropic Sobolev spaces, see the following lemma,
being along the lines of \cite[Theorem 2.1]{Lions1968}. It has been proven 
for $\Gamma_0 \in C^\infty$, but it is also true for a Lipschitz boundary
in accordance with \cite[pg.~504ff]{Costabel1990}.

\begin{lemma}
The map
\[
 \gamma_0 \colon H^{1,\frac{1}{2}} (Q_0) 
 	\to H^{\frac{1}{2}, \frac{1}{4}} (\Sigma_0)
\]
is linear and continuous. 
\end{lemma}

We find the following statement in \cite[Lemma 2.4]{Costabel1990}, 
which holds in the case of a Lipschitz domain $\Omega_0$.

\begin{lemma}
	\label{lem.Dirichlet_trace_cont_and_surj}
The Dirichlet trace operator $\gamma_0$ is continuous and 
surjective as an operator from $H^{1,\frac{1}{2}}_{;0,} (Q_0)$ to 
$H^{\frac{1}{2},\frac{1}{4}} (\Sigma_0)$. 
\end{lemma}

According to \cite[Theorem 2.4]{Dohr2019a}, there exists also an extension 
operator. The extension operator is a right inverse to the surjective Dirichlet trace 
operator $\gamma_0$ and, thus, extends a function defined only 
on the boundary to the space (see also \cite[pg.~12]{Dohr2019} and \cite[Definition 
2.17]{Costabel1990}). 

\begin{lemma}\label{lem.extension_of_trace}
The Dirichlet trace operator
\[ 
  \gamma_0 \colon H^{1, \frac{1}{2}}_{;0,}(Q_0) 
  	\to H^{\frac{1}{2},\frac{1}{4}} (\Sigma_0)
\]
has a continuous right inverse operator
\[
  \Ecal_0\colon H^{\frac{1}{2}, \frac{1}{4}} (\Sigma_0) \to H^{1, \frac{1}{2}}_{;0,} (Q_0),
\]
satisfying $\gamma_0 \Ecal_0 v = v$ for all 
$v \in H^{\frac{1}{2},\frac{1}{4}}(\Sigma_0)$.
\end{lemma}

\subsection{Dirichlet trace operator on non-cylindrical domains}
\label{subsec.trace_op_non_cylindrical_domains}
In this section, we denote the (interior) Dirichlet trace operator with respect 
to a non-cylindrical domain by $\gamma_{0,t}$ to 
distinguish it from the Dirichlet trace operator with respect to a cylindrical 
domain introduced above. When no confusion can happen, we will drop 
the subscript $t$ in the trace operator for a non-cylindrical domain. 

For a smooth function $u\in C^1(\overline{Q}_T)$, defined on a 
non-cylindrical domain, we set
\[ 
\gamma_{0,t} u (t, \xbfm_t) \isdef 
	\lim_{\Omega_t \ni \ybfm_t \to \xbfm_t \in \Gamma_t} u (t, \ybfm_t ).
\]
It obviously holds
\[ 
  \gamma_{0,t} u (t, \xbfm_t)  
    = \lim_{\substack{\Omega_0 \ni \ybfm \to \xbfm \in \Gamma_0, \\ \kappabfm(t,\xbfm) 
    = \xbfm_t}} u \big(t, \kappabfm (t,\ybfm) \big) = \gamma_{0}(u \circ \kappabfm) (t, \xbfm) 
    = \gamma_{0}(u \circ \kappabfm) \big(t, \kappabfm^{-1} (t,\xbfm_t) \big)
\]
for the diffeomorphism $\kappabfm$ from \eqref{eq:deformation}. By density 
of the smooth functions in the Sobolev spaces, we can also extend this 
notion to Sobolev spaces. Moreover, we have the same mapping properties 
for $\gamma_{0,t}$ as for $\gamma_0$, 
since
\begin{align*}
\| \gamma_{0,t} u \|_{H^{\frac{1}{2},\frac{1}{4}} (\Sigma_T)} 
	& = \| \gamma_{0,t} u\circ \kappabfm \|_{H^{\frac{1}{2},\frac{1}{4}}(\Sigma_0)} \\
	& = \big\| \big(\gamma_{0}(u \circ \kappabfm)  \circ \kappabfm^ {-1} \big)
			\circ \kappabfm \big\|_{H^{\frac{1}{2},\frac{1}{4}}(\Sigma_0)} \\
	& \lesssim \| u \circ \kappabfm \|_{H^{1, \frac{1}{2}}(Q_0)} \\
	& = \| u \|_{H^{1, \frac{1}{2}}(Q_T)}.
\end{align*}
Note that the hidden constant changes from line to line and depends 
on the diffeomorphism $\kappabfm$, because we used the norm equivalence 
on the cylindrical and non-cylindrical domain as well as the mapping property 
of the Dirichlet trace operator on the cylindrical domain.

Due to this consideration, all the properties of Section \ref{subsec.trace_op_cylindrical} 
remain valid for the Dirichlet trace operator on non-cylindrical domains. The surjectivity 
follows for example from the following consideration: Let $v \in H^{\frac12, \frac14} 
(\Sigma_T)$. By the definition of the norm, we thus have $v \circ \kappabfm \in 
H^{\frac12, \frac14} (\Sigma_0)$. By the surjectivity of the Dirichlet trace 
operator with respect to $Q_0$, there exists a $w\in H^{1, \frac12}_{;0,} (Q_0)$ 
with $\gamma_{0}w = v \circ \kappabfm$. Due to the bijectivity of 
$\kappabfm$, we may define 
\[
\hat{w}\isdef w\circ \kappabfm^{-1}\in H^{1, \frac12}_{;0,}(Q_T).
\] 
We thus have 
\[
\gamma_{0,t} \hat{w}(t, \xbfm_t) 
= \gamma_0 w \big(t, \kappabfm^{-1} (t,\xbfm_t)\big) 
= v (t,\xbfm_t), 
\]
from where the subjectivity follows and we can also 
infer the existence of the right inverse operator $\Ecal_0$.

\subsection{Existence and uniqueness of Dirichlet problem}
\label{sec.ex_and_unique_Dirichlet}
We consider the following Dirichlet problem with homogeneous initial datum
\begin{equation}\label{eq.generalized_diff_eq}
\begin{aligned}
(\partial_t - \Delta) v & = f   && \ \ \text{in } Q_T,  \\
\gamma_0 v & = q && \ \ \text{on } \Sigma_T, \\
v(0, \cdot) & = 0   &&\ \ \text{in } \Omega_0.
\end{aligned}
\end{equation}
We have the following existence and uniqueness theorem for
its solution.

\begin{theorem}\label{thm.solution_operator_isomorphism}
Let $f \in H^{-1, -\frac{1}{2}}_{;0,}(Q_T)$ and $q \in H^{\frac{1}{2},\frac{1}{4}}(\Sigma_T)$.
Then, there exists a unique solution $v \in H^{1,\frac{1}{2}}_{;0,} (Q_T)$,
satisfying the boundary condition in \eqref{eq.generalized_diff_eq} and
\begin{equation}\label{eq.weak_form_pde}
  S(v, \varphi) \isdef \int_0^T\int_{\Omega_t} \{\nabla v \cdot \nabla \varphi + \partial_t v \varphi\}\, \drm \xbfm \drm t
	= \int_0^T \int_{\Omega_t} f u \, \drm \xbfm \drm t \ \text{for all}\ \varphi \in H^{1, \frac{1}{2}}_{0;,0}(Q_T).
\end{equation}
\end{theorem}

\begin{proof}
A proof of this statement based on the proof of \cite[Lemma 2.8]{Costabel1990} 
can be found in \cite{Bruegger2020}.
\end{proof}

\section{Neumann trace operator}
\label{sec.neumann_trace_op}
Similarly as we defined the Dirichlet trace operator, 
we can also introduce an (interior) Neumann trace 
operator. In the following, we will first introduce this 
concept on cylindrical domains. Then, we will introduce 
the notion of a Neumann trace on a non-cylindrical domain 
formally and rigorously. 

\subsection{Neumann trace operator on cylindrical domains}
\label{subsec.neumann_trace_cylindrical}
We first introduce the Neumann trace operator, also called the 
conormal derivative, on a cylindrical domain along the lines 
of \cite{Costabel1990}. Let us define the space
\[
H^{1,\frac12} (Q_0; \Lcal) \isdef \big\lbrace u \in H^{1, \frac12} (Q_0) \colon \Lcal u \in L^2(Q_0) \big\rbrace,
\]
where $\Lcal \isdef \partial_t - \Delta$ is the partial differential 
operator under consideration. The norm on this space is given by
\[ 
  \|u \|_{H^{1, \frac12}(Q_0; \Lcal)}^2 \isdef \|u \|_{H^{1, \frac12} (Q_0)}^2 + \big\| (\partial_t - \Delta) u \big\|_{L^2 (Q_0)}^2. 
\]

According to \cite[Lemma 2.16]{Costabel1990}, the bilinear form
\[
b(u,v) \isdef \int_{Q_0} \big\{\langle \nabla u, \nabla v \rangle 
	- (\partial_t - \Delta) u v\big\}\, \drm \xbfm \drm t + d(u,v)
\]
is continuous on $H^{1, \frac12}(\Rbbb \times \Omega_0; \Lcal) \times 
H^{1, \frac12} (\Rbbb \times \Omega_0)$, where 
\[ 
  d (u,v) \isdef \int_{\Rbbb \times \Omega_0} \partial_t u v \, \drm \xbfm \drm t.
\]
The bilinear form $d(u,v)$ has a continuous extension 
from $C_0^\infty (\Rbbb^{d+1})\times C_0^\infty (\Rbbb^{d+1})$ to 
$H^\frac12 \big(\Rbbb; L^2 (\Omega_0) \big)\times H^\frac12\big(\Rbbb; 
L^2 (\Omega_0) \big)$ and it holds $d(u,v) = - d(v,u)$ for all 
$u,v \in H^{\frac12} \big(\Rbbb; L^2 (\Omega_0) \big)$, compare 
\cite[Lemma 2.6]{Costabel1990}. 

The (interior) Neumann trace is defined for $u \in C^1 (\overline{Q}_0)$ by
\[ 
  \gamma^{\operatorname{int}}_1 u (t, \xbfm) 
  	\isdef \lim_{\Omega_0 \ni \ybfm \to \xbfm \in \Gamma_0} \big\langle \nabla_{\ybfm} u (t, \ybfm), 
		\nbfm_{\xbfm} \big\rangle \quad \text{for } (t, \xbfm) \in \Sigma_0  
\]
and coincides with the normal derivative on $\Sigma_0$, thus 
$\gamma_1^{\operatorname{int}} u = \partial u/\partial \nbfm$ on 
$\Sigma_0$, see \cite[Section 3.3]{Dohr2019a} and also 
\cite[Satz 8.7]{Wloka1987} for the elliptic case. Since it holds
\[ 
  b(u,v) = \int_{\Sigma_0} \frac{\partial u}{\partial {\bf n}} v \,  \drm \sigma \drm t
\]
for $u,v \in C_0^2 (\Rbbb \times\overline{\Omega}_0)$, 
we can extend this definition as follows, which is along the 
lines of \cite[Definition 2.17]{Costabel1990}.

\begin{definition}
Let $u \in H^{1, \frac12} (\Rbbb \times \Omega_0; \Lcal)$. Then, 
the Neumann trace operator $\gamma_1 u \in H^{-\frac12, -\frac14} (\Sigma_0)$ 
is the continuous linear form on $H^{\frac12, \frac12} (\Sigma_0)$ defined by
\[ 
  \gamma_1^{\operatorname{int}} u \colon \varphi \mapsto b (u, \Ecal_0 \varphi),
\]	
where $\Ecal_0$ is the extension operator given in Lemma \ref{lem.extension_of_trace}.
\end{definition}

Notice that we can also introduce the conormal derivative 
$\gamma_1^{\operatorname{int}} u \in H^{-\frac12, -\frac14}(\Sigma_0)$ 
as the unique solution of a variational problem, as it is done in \cite[Section 3.4]{Dohr2019}. 
This variational problem can for example be obtained by applying $\gamma_1^{\operatorname{int}} 
u$ to $\varphi$. According to \cite[Proposition 2.18]{Costabel1990}, the Neumann trace 
has the following properties.

\begin{lemma}
\begin{enumerate}[(i)]
\item The map
\[ 
\gamma_1^{\operatorname{int}}\colon  H^{1, \frac12} (\Rbbb \times \Omega_0; \Lcal) 
	\to H^{-\frac12, -\frac14} (\Rbbb \times \Gamma_0)
\]
is continuous and by restriction also the map
\[ 
  \gamma_1^{\operatorname{int}}\colon H^{1, \frac12} (Q_0; \Lcal) \to H^{-\frac12, -\frac14} (\Sigma_0) 
\]
is continuous.
\item If $u \in C^2 (\overline{Q}_0)$, then $\gamma_1^{\operatorname{int}} u 
= (\partial u/\partial \nbfm)|_{\Sigma_0}$ due to the Green formula.
\end{enumerate}
\end{lemma}

\subsection{Neumann trace operator on non-cylindrical domains}
\label{subsec.neumann_trace_tube}
On time-de\-pen\-dent boundaries, one could consider the usual Neumann 
trace, as it is done for example in \cite[Section 6.1]{Dziri2001}. Instead, we
follow here the idea of \cite{Tausch2019} and employ a velocity corrected Neumann 
trace. We first formally introduce this Neumann trace and afterwards characterize 
its properties rigorously.
\subsubsection{Formal}
For a time dependent spatial surface we define two
Neumann trace operators
\begin{equation}\label{eq.normaltrace}
 \gamma_1^\pm \varphi \isdef \frac{\partial \varphi}{\partial \nbfm_t} 
 	\mp \frac{1}{2} \langle \Vbfm, \nbfm_t \rangle \varphi. 
\end{equation}
To motivate this definition
consider the boundary value problem
\begin{equation}\label{eq.diff_eq_Neumann1}
\begin{aligned}
(\partial_t - \Delta) u & = f   && \ \ \text{in } Q_T,  \\
\gamma_1 u & = g && \ \ \text{on } \Sigma_T, \\
u(0, \cdot) & = 0   &&\ \ \text{in } \Omega_0,
\end{aligned}
\end{equation}
where we leave it a priori open what $\gamma_1$ means.
Let us formally derive the weak formulation of the Neumann 
problem \eqref{eq.diff_eq_Neumann1} by multiplying with a 
test function $v$ with $v(T, \cdot) = 0$ in $\Omega_T$ and 
using Reynolds' transport theorem
\begin{align*}
\int_0^T \int_{\Omega_t} f v \, \drm \xbfm \drm t 
&= \int_0^T \int_{\Omega_t} (\partial_t  - \Delta ) u v \, \drm \xbfm \drm t \\
&= \int_0^T \int_{\Omega_t} \big\{\langle \nabla u, \nabla v \rangle + \partial_t (uv) - u \partial_t v\big\} \, \drm \xbfm \drm t - \int_0^T \int_{\Gamma_t} \frac{\partial u}{\partial \nbfm} v \, \drm \sigma \drm t \\
&= \int_0^T \int_{\Omega_t} \big\{\langle \nabla u, \nabla v \rangle  - u \partial_t v\big\} \, \drm \xbfm \drm t 
+ \int_0^T \frac{\drm}{\drm t} \int_{\Omega_t} uv \, \drm \xbfm \drm t\\
& \hspace*{4.5cm} - \int_0^T \int_{\Gamma_t} \bigg\{\frac{\partial u}{\partial \nbfm} v 
+ u v \langle \Vbfm, \nbfm \rangle\bigg\}  \, \drm \sigma \drm t.
\end{align*}
Due to the fundamental theorem of calculus and the vanishing 
initial and end condition of $u$ and $v$, respectively, we obtain
the variational equation
\[ 
a(u,v) = \int_0^T \int_{\Omega_t} f v \, \drm \xbfm \drm t 
+ \int_0^T \int_{\Gamma_t} \bigg\{\frac{\partial u }{\partial \nbfm} v 
+ \frac12 uv \langle \Vbfm, \nbfm \rangle\bigg\} \, \drm \sigma \drm t 
\]
with bilinear form
\[ 
a(u,v) \isdef \int_0^T \int_{\Omega_t}\big\{\langle \nabla u, \nabla v \rangle 
- u \partial_t v\big\} \, \drm \xbfm \drm t 
- \frac{1}{2}\int_0^T \int_{\Gamma_t} uv \langle \Vbfm, \nbfm \rangle \, \drm \sigma \drm t.\]
Thus, if we set the previously unspecified trace in \eqref{eq.diff_eq_Neumann1} 
as $\gamma_1^{-}$, and we arrive at
\[ 
a(u,v) = \int_0^T \int_{\Omega_t} f v \, \drm \xbfm \drm t 
+ \int_0^T \int_{\Gamma_t} g v \, \drm \sigma \drm t.
\]
With the additonal boundary integral term
  it is easy to see that the bilinear form is bounded
  and coercive in the $H^{1, \frac12}$-norm, so
the Lion's projection theorem guarantees
existence and uniqueness of this Neumann problem. However,
for the cylindrical case \cite[Lemma 2.21]{Costabel1990} 
states that this strategy does not yield satisfactory results, 
since one has to make stronger assumptions on the regularity
of the input data.
Therefore, as in \cite{Costabel1990}, we will proof the existence 
and uniqueness of solutions by using a boundary integral 
formulation (see Corollary \ref{cor.ex_and_unique_Neumann}).

\subsubsection{Rigorous}
We assume $\kappabfm$ to be defined on $\Rbbb \times \Rbbb^d$ 
and not only on $[0,T] \times \Rbbb^d$. Moreover, for sake 
of simplicity in representation, we always consider functions 
$u$ and $v$ throughout this section which satisfy
\begin{equation}\label{eq:assumption}
\int_{\Rbbb} \frac{\drm}{\drm t} \int_{\Omega_t} u v \, \drm \xbfm \drm t = 0.
\end{equation}
This assumption stems from the fact that we would like to 
integrate by parts in time. Later on, we will consider a finite time 
interval $(0,T)$ and equip $u$ and $v$ with the appropriate zero 
initial and end conditions. Extending $u$ and $v$ by zero for 
$t<0$ and $t>T$, respectively, leads then to the fulfillment of 
\eqref{eq:assumption}.

Let us define
\begin{equation}
\label{eq.def_duv}
d(u,v) \isdef \int_{\Rbbb} \int_{\Omega_t} \partial_t u v \, \drm \xbfm \drm t 
+ \frac12  \int_{\Rbbb} \int_{\Gamma_t} \langle \Vbfm, \nbfm \rangle u v\,\drm \sigma \drm t. 
\end{equation}
Notice that the additional boundary term is a speciality 
of the time-dependent boundary. We shall first state 
the analogue of \cite[Lemma 2.6]{Costabel1990}.

\begin{lemma}
The bilinear form $d(u,v)$ has a continuous extension from $C_0^{\infty} 
(\Rbbb^{d+1})\times C_0^{\infty} (\Rbbb^{d+1})$ to $H^{1,\frac12}
\big(\bigcup_{t \in \Rbbb}(\lbrace t \rbrace \times \Omega_t)\big)\times 
H^{1,\frac12}\big(\bigcup_{t \in \Rbbb}(\lbrace t \rbrace\times\Omega_t)\big)$, 
and it holds
\begin{equation}\label{eq.duv_eq_m_dvu}
  d(u,v) = - d(v,u). 
\end{equation}
\end{lemma}

\begin{proof}
The use of Reynolds' transport theorem allows us to compute
\begin{align*}
  d(u,v) &= \int_{\Rbbb} \int_{\Omega_t} \partial_t u v \, \drm \xbfm \drm t 
  	+ \frac12 \int_{\Rbbb} \int_{\Gamma_t} \langle \Vbfm, \nbfm \rangle u v\,\drm \sigma \drm t \\
  &= \int_{\Rbbb} \int_{\Omega_t} \big\{\partial_t (uv) - u \partial_t v\big\} \, \drm \xbfm \drm t  
	+ \frac12  \int_{\Rbbb} \int_{\Gamma_t} \langle \Vbfm, \nbfm \rangle u v\,\drm \sigma \drm t \\
  &= \int_{\Rbbb} \frac{\drm}{\drm t} \int_{\Omega_t} u v \, \drm \xbfm \drm t 
 	- \int_{\Rbbb} \int_{\Omega_t} u \partial_t v \, \drm \xbfm \drm t
  	- \frac12 \int_{\Rbbb} \int_{\Gamma_t} \langle \Vbfm, \nbfm \rangle u v \, \drm \sigma \drm t.
\end{align*}
The assumption \eqref{eq:assumption} hence implies
\[
  d(u,v) = - \int_{\Rbbb} \int_{\Omega_t} u \partial_t v \, \drm \xbfm \drm t 
  	- \frac12 \int_{\Rbbb} \int_{\Gamma_t} \langle \Vbfm, \nbfm \rangle u v \, \drm \sigma \drm t, 
\]
from where \eqref{eq.duv_eq_m_dvu} follows immediately. The rest 
is in complete analogy to \cite[Lemma 2.6]{Costabel1990}, but we need 
higher regularity in the spatial variable instead just $L^2(\Omega_0)$
like in \cite{Costabel1990}, because the boundary term in the definition 
of $d(u,v)$ has to be well-defined. 
\end{proof}

As in Section \ref{subsec.neumann_trace_cylindrical}, we introduce the space 
\[ 
H^{1, \frac12} (Q_T; \Lcal) \isdef \big\lbrace u \in 
H^{1, \frac12} (Q_T) \colon \Lcal u \in L^2 (Q_T) \big\rbrace,
\] 
where $\Lcal \isdef \partial_t - \Delta$ is the differential operator 
on the non-cylindrical domain. We state the analogue of 
\cite[Lemma 2.16]{Costabel1990} in the case of a non-cylindrical 
domain, the proof of which is obvious.

\begin{lemma}\label{lem.used_for_neumann_trace}
The bilinear form
\[ 
  b^-(u,v) \isdef \int_0^T \int_{\Omega_t} \big\{\langle \nabla u, \nabla v \rangle 
  	- (\partial_t - \Delta) u v\big\} \, \drm \xbfm \drm t + d (u,v)
\]
with $d(u,v)$ being defined in \eqref{eq.def_duv} is continuous 
on $H^{1,\frac12}\big( \bigcup_{t \in \Rbbb}(\lbrace t \rbrace 
\times \Omega_t); \partial_t - \Delta\big) \times H^{1, \frac12} \big( \bigcup_{t \in \Rbbb} 
(\lbrace t \rbrace \times \Omega_t) \big)$. If $u, v \in C_0^2 \big( \bigcup_{t \in \Rbbb} 
(\lbrace t \rbrace \times \overline{\Omega}_t)\big)$, we have
\[ 
  b^-(u,v) = \int_0^T \int_{\Gamma_t} \bigg\{\frac{\partial u}{\partial \nbfm} v 
  	+ \frac12 \langle \Vbfm, \nbfm \rangle u v\bigg\} \, \drm \sigma \drm t 
\]
by means of Green's formula.
\end{lemma}

In complete analogy to the Neumann trace operator in 
the cylindrical case, we will define $\gamma_1^- u$, which is 
one of the two required Neumann trace operators.

\begin{definition}\label{def.gamma_1_minus}
Given $u \in H^{1, \frac12}\big( \bigcup_{t \in \Rbbb}(\lbrace t \rbrace \times \Omega_t); 
\partial_t - \Delta \big)$, we denote by $\gamma_1^- u \in H^{-\frac12, -\frac14} (\Sigma_T) $ 
the continuous linear form on $H^{\frac12, \frac14} (\Sigma_T)$ defined by
\[ 
\gamma_1^- u\colon \varphi \mapsto b^-(u, \Ecal_0 \varphi),
\]
where $\Ecal_0$ is the extension operator as mentioned in 
Section \ref{subsec.trace_op_non_cylindrical_domains}. 
\end{definition}

The following lemma is the non-cylindrical equivalent to 
\cite[Proposition 2.18]{Costabel1990}.

\begin{lemma}\label{lem.mapping_properties_gamma_1_minus}
The map 
\[ 
  \gamma_1^- \colon H^{1, \frac12} \bigg( \bigcup_{t \in \Rbbb} \big(\lbrace t \rbrace \times \Omega_t \big); 
  	\partial_t - \Delta \bigg) \to H^{-\frac12, -\frac14}  
		\bigg( \bigcup_{t \in \Rbbb} \big(\lbrace t \rbrace \times \Gamma_t \big)\bigg)
\]
is continuous and by restriction also the map
\[ 
  \gamma_1^- \colon H^{1, \frac12} (Q_T; \partial_t - \Delta)\to H^{-\frac12, -\frac14} (\Sigma_T)
\]
is continuous. Moreover, if $u \in C^2 (\overline{Q}_T)$, then it holds 
\[
\gamma_1^- u = \frac{\partial u}{\partial \nbfm} + \frac12 \langle \Vbfm, \nbfm \rangle u.
\]
\end{lemma}

\begin{proof}
As in \cite{Costabel1990}, the continuity is a consequence of 
the continuity of the bilinear form $b(\cdot, \cdot)$ (cf.~Lemma 
\ref{lem.used_for_neumann_trace}). The second statement follows 
immediately from Green's first formula.
\end{proof}

\begin{remark}
In view of the reformulation of the heat equation in terms of 
boundary integral equations, we will moreover encounter a 
second Neumann trace operator, which we denote by 
$\gamma_1^+$. It can be achieved analogously to 
above by considering the differential operator 
$\partial_t + \Delta$ instead of $\partial_t - \Delta$. The 
former operator for example arises when considering a 
time reversal of the latter one. With
\[ 
b^{+ }(u,v) \isdef \int_0^T \int_{\Omega_t} \big\{\langle \nabla u, \nabla v \rangle 
	+ (\partial_t + \Delta) u v\big\} \, \drm \xbfm \drm t - d(u,v), 
\]
we can state the analogue of Lemma \ref{lem.used_for_neumann_trace}, 
namely the continuity of $b^{+}(\cdot, \cdot)$ in the appropriate space and 
for $u$ and $v$ smooth enough, we have
\[ 
  b^{+} (u,v) = \int_0^T \int_{\Gamma_t} \bigg\{\frac{\partial u}{\partial \nbfm} v 
  - \frac12  \langle \Vbfm, \nbfm \rangle u v\bigg\} \, \drm \sigma \drm t. 
\]
With this property at hand, we can define the trace operator 
$\gamma_1^+$ in analogy to Definition \ref{def.gamma_1_minus}. 
For $u$ smooth enough, it then holds 
\[ 
  \gamma_1^+ u = \frac{\partial u}{\partial \nbfm} - \frac12 \langle \Vbfm, \nbfm \rangle u.
\]
The existence of two Neumann trace operators is a speciality of 
the time-dependent boundary.
\end{remark}

Likewise to \cite[Formula (2.35)]{Costabel1990}, given
$u\in H^{1, \frac12}\big(\bigcup_{t \in \Rbbb}(\lbrace t \rbrace \times \Omega_t); 
\partial_t - \Delta\big)$ and $v \in H^{1, \frac12}\big( \bigcup_{t \in \Rbbb}
(\lbrace t \rbrace \times \Omega_t)\big)$, we obtain Green's first formula
\begin{equation}\label{eq.first_green_formula}
   \int_{\Rbbb} \int_{\Omega_t}\langle \nabla u, \nabla v \rangle\, \drm \xbfm \drm t + d(u,v) 
   = \int_{\Rbbb} \int_{\Omega_t} (\partial_t - \Delta) u v \, \drm \xbfm \drm t 
   + \langle \gamma_1^- u, \gamma_0 v \rangle. 
\end{equation}
By restriction, this formula also holds for $u \in H^{1, \frac12 }_{;0,} 
(Q_T; \partial_t - \Delta)$ and $v \in H^{1, \frac12}_{;,0} (Q_T)$,
but not, as was pointed out in \cite{Costabel1990}, when $u,v$
are both in $H^{1,\frac12}_{;0,} (Q_T)$.

In complete analogy, Green's formula for $u \in H^{1, \frac12}\big( \bigcup_{t \in \Rbbb} 
(\lbrace t \rbrace \times \Omega_t); \partial_t +\Delta\big)$ and $v \in
H^{1, \frac12}\big(\bigcup_{t \in \Rbbb}(\lbrace t \rbrace\times \Omega_t)\big)$ reads 
\[ 
\int_{\Rbbb} \int_{\Omega_t}\langle \nabla u, \nabla v \rangle\, \drm \xbfm \drm t - d(u,v) 
= \int_{\Rbbb} \int_{\Omega_t} (- \partial_t - \Delta) u v \, \drm \xbfm \drm t + \langle \gamma_1^+ u, \gamma_0 v \rangle. 
\]
Again, by restriction, this formula also holds for $u \in H^{1, \frac12 }_{;,0} 
(Q_T; \partial_t + \Delta)$ and $v \in H^{1, \frac12}_{;0,} (Q_T)$.

We can now formulate Green's formulas for a finite time interval, the 
time-independent analogues of which are given in \cite[Proposition 2.19]{Costabel1990}. 

Notice that \cite{Costabel1990} introduces a time reversal map. For a time-dependent 
domain, this approach does not make sense, since the integration over a time forward 
tube of a time reversed entity is not always well defined. Therefore, we choose a slightly 
different approach to obtain a further Green formula.

\begin{lemma}\label{lem.green_formula_tube}~
\begin{enumerate}[(i)]
\item Let $u \in H^{1, \frac12}_{;0,}\big( \bigcup_{t \in \Rbbb_+}(\lbrace t \rbrace \times \Omega_t); 
\partial_t - \Delta\big)$ and $v \in H^{1, \frac12}_{;,0}\big(\bigcup_{- \infty < t < t_0}(\lbrace t \rbrace \times \Omega_t)\big)$. 
Then, for $t_0>0$, there holds Green's first formula
\[
  \int_0^{t_0} \int_{\Omega_t}\langle \nabla u, \nabla  v\rangle 
  	\, \drm \xbfm \drm t + d  (u,  v) = \int_0^{t_0} \int_{\Omega_t} (\partial_t - \Delta) u v \, \drm \xbfm \drm t 
  	+ \langle \gamma_1^- u, \gamma_0  v \rangle.
\]
\item Let $u \in H^{1, \frac12}_{;,0}\big( \bigcup_{- \infty < t < t_0}(\lbrace t \rbrace \times \Omega_t); 
\partial_t + \Delta\big)$ and $v \in H^{1, \frac12}_{;0,}\big(\bigcup_{t \in \Rbbb_+}(\lbrace t \rbrace \times \Omega_t)\big)$. 
Then, for $t_0>0$, there holds Green's alternative first formula
\[ 
\int_0^{t_0} \int_{\Omega_t}\langle \nabla u, \nabla v \rangle\, \drm \xbfm \drm t - d(u,v) 
= \int_0^{t_0} \int_{\Omega_t} (- \partial_t - \Delta) u v \, \drm \xbfm \drm t + \langle \gamma_1^+ u, \gamma_0 v \rangle. 
\]

\item Let $u \in H^{1, \frac12}_{;0,}\big( \bigcup_{t \in \Rbbb_+}(\lbrace t \rbrace \times \Omega_t); 
\partial_t - \Delta\big)$ and $v \in H^{1, \frac12}_{;,0}\big( \bigcup_{- \infty < t < t_0}(\lbrace t \rbrace \times \Omega_t); 
\partial_t + \Delta\big)$.
Then, for $t_0>0$, there holds Green's second formula 
\[
  \int_0^{t_0} \int_{\Omega_t} \big\{(\partial_t - \Delta) u  v 
  	+ u (\partial_t + \Delta) v\big\} \, \drm \xbfm \drm t = \langle \gamma_0  u, \gamma_1^+  v \rangle 
	- \langle \gamma_1^- u, \gamma_0  v \rangle.
\]
\end{enumerate}
\end{lemma}
\begin{proof}
Statements (i) and (ii) are clear. Statement (iii) follows then
immediately from these by interchanging $v$ and $u$ in (ii) 
and using \eqref{eq.duv_eq_m_dvu}.
\end{proof}

We need the tube equivalent of \cite[Lemma 2.22]{Costabel1990}. In there, the space
\[ \widetilde{C}^{\infty} (\overline{Q}_0) \isdef C_0^\infty \big((0,T] \times\overline{\Omega}_0 \big)\]
is defined as the space of the restrictions of functions in $C_0^\infty (\Rbbb_+ \times \Rbbb^d)$ 
to $\overline{Q}_0$. This space $\widetilde{C}^\infty (\overline{Q}_0)$ is dense in $H^{1, \frac12}_{;0,} 
(Q_0; \partial_t - \Delta)$ according to \cite[Lemma 2.22]{Costabel1990}. As we only consider a 
$C^2$-mapping between the reference cylinder and the tube, we will prove the analogue 
result only for $C^2$-functions. 

\begin{lemma}\label{lem.density_C2_H1}
Let us define
\[ 
  \widetilde{C}^2 (\overline{Q}_T) \isdef \Big\lbrace u \colon u \circ \kappabfm 
  	\in C^2_0 \big( (0,T] \times \overline{\Omega}_0 \big) \Big\rbrace.
\]
Then, the space $\widetilde{C}^2 (\overline{Q}_T)$  is dense in 
$H^{1, \frac12}_{;0,} (Q_T; \partial_t - \Delta)$. 
\end{lemma}

\begin{proof}
We mimic the proof of \cite[Lemma 2.22]{Costabel1990}, which is based on 
a proof of Grisvard in the elliptic case, see \cite[Lemma 1.5.3.9]{Grisvard1985}.
According to the proof of \cite[Lemma 2.22]{Costabel1990}, we have that 
$C_0^\infty \big( (0,T] \times \overline{\Omega}_0 \big)$ is dense in 
$H^{1,\frac12}_{;0,} (Q_0)$. Therefore, also $C_0^2 \big((0,T] \times 
\overline{\Omega}_0 \big)$ is dense in $H^{1,\frac12}_{;0,} (Q_0)$. 
Due to the definition of the spaces on the tube via the mapping 
$\kappabfm$ and the resulting equivalence of norms, we also obtain 
that $\widetilde{C}^2 (\overline{Q}_T) $ is dense in $H^{1,\frac12}_{;0,} (Q_T)$. 
Similarly, we obtain that 
\[
\widetilde{C}^2 (Q_T) \isdef \Big\lbrace u \colon u \circ \kappabfm 
\in C_0^2 \big( (0,T] \times \Omega_0 \big) \Big\rbrace
\]
is dense in $H^{1,\frac12}_{0;0,} (Q_T)$. 

Let $\Rcal$ be an extension operator from $H^{1, \frac12}_{;0,} (Q_T)$
to $H^{1, \frac12} (\Rbbb^{d+1})$. It thus holds $(\Rcal u)|_{Q_T} = u$. 
As in \cite{Costabel1990}, let us choose $\Rcal$ such that $\supp\Rcal
\subset [0,\infty) \times \Rbbb^d$.\footnote{Such an extension operator exists
as it can be defined by $\Rcal u = \big(\widetilde{\Rcal}(u \circ \kappabfm)\big) 
\circ\kappabfm^{-1}$ with $\widetilde{\Rcal}: H^{1, \frac12}_{;0,}(Q_0)\to 
H^{1,\frac12}(\Rbbb_+ \times \Rbbb^d)$ being the extension operator from 
\cite{Costabel1990}.} In that way, we can identify $H^{1, \frac12}_{;0,} (Q_T)$ 
with a closed subspace of $H^{1, \frac12}_{;0,} (\Rbbb_+ \times \Rbbb^d)$. 
The map $u \mapsto \big( \Rcal u, (\partial_t - \Delta) u \big)$ identifies 
$H^{1, \frac12}_{;0,} (Q_T; \partial_t - \Delta)$ with a closed subspace 
of $H^{1, \frac12}_{;0,} (\Rbbb_+ \times \Rbbb^d) \times L^2 (\Rbbb_+ 
\times \Rbbb^d)$. Due to this identification, we find for every bounded 
linear functional $\ell:H^{1, \frac12}_{;0,} (Q_T; \partial_t - \Delta)\to\Rbbb$ 
some $f \in \big( H^{1, \frac12}_{;0,} (\Rbbb_+ \times \Rbbb^d) \big)'
= H^{-1,-\frac12} (\Rbbb_+ \times \Rbbb^d)$ and $g \in L^2 
(\Rbbb_+ \times \Rbbb^d)$ such that it holds
\[ 
  \langle \ell, u \rangle = \langle f, \Rcal u \rangle + \int_{\Rbbb_+ \times \Rbbb^d} 
  	g (\partial_t - \Delta) u \, \drm(t,\xbfm)
\]
for all $u \in H^{1, \frac12}_{;0,} (Q_T; \partial_t - \Delta)$. Since $\ell$ 
acts only on $u$, which is supported on $\overline{Q}_T$, we may assume
that $\supp f \subset\overline{Q}_T$ and $\supp g \subset \overline{Q}_T$. 

We shall suppose next that it holds $\langle \ell, \varphi \rangle = 0$ for all 
$\varphi \in \widetilde{C}^2 (\overline{Q}_T)$. If we can show $\ell = 0$, 
we obtain the desired density result in accordance with \cite[Korollar III.1.9]
{Werner2018}. For all $\varphi \in C_0^2(\Rbbb_+ \times \Rbbb^d)$, we conclude
\begin{align*}
0 = \langle \ell, \varphi \rangle 
&= \langle f, \varphi \rangle  + \int_{Q_T} g (\partial_t - \Delta) \varphi \,\drm(t,\xbfm)\\
& = \langle f, \varphi \rangle + \int_{\Rbbb_+ \times \Rbbb^d} g(\partial_t - \Delta) \varphi \,\drm(t,\xbfm).
\end{align*}
This equation states that 
\[ 
  f = (\partial_t + \Delta) g
\]
holds on $\Rbbb_+ \times \Rbbb^d$ in complete analogy to \cite{Costabel1990}. 
Due to $f \in H^{-1,-\frac12} (\Rbbb_+ \times \Rbbb^d)$ and the differential operator, 
we find $g \in H^{1, \frac12} (\Rbbb_+ \times \Rbbb^d)$ and, thus, $g|_{Q_T} \in 
H^{1, \frac12}_{0;,0} (Q_T)$. 

On a cylindrical domain, any function $h\in H^{1, \frac12}_{0;,0} (Q_0)$ 
can be approximated by a series $h_n \in C_0^\infty \big( (-\infty ,T) \times \Omega_0\big)$ 
(see \cite[Proof of Lemma 2.22]{Costabel1990}). Hence, by choosing $h\isdef g\circ\kappabfm$
and setting $g_n \isdef h_n\circ\kappabfm^{-1}$, we can approximate $g|_{Q_T}\in H^{1, \frac12}_{0;,0}(Q_T)$ 
by a series $g_n\in C_0^2 \big(\bigcup_{-\infty < t< T}(\lbrace t \rbrace \times \Omega_t)\big)$ 
in the norm of $H^{1, \frac12}\big(\bigcup_{0 < t< \infty}(\lbrace t \rbrace \times \Omega_t)\big)$.
Thus, denoting by $\hat{g}_n$ the extension of $g_n$ by zero outside of $Q_T$, we find
$(\partial_t + \Delta) \hat{g}_n \to f$ in $H^{-1, -\frac12} (\Rbbb_+ \times \Rbbb^d)$. 
We then conclude for any $u \in H^{1, \frac12}_{;0,} (Q_T; \partial_t - \Delta)$ that
\begin{align*}
  \langle \ell, u \rangle &= \lim_{n \to \infty} \left[ \big\langle (\partial_t + \Delta) \hat{g}_n, 
  	\Rcal u \big\rangle + \int_{Q_T} g_n (\partial_t - \Delta) u \, \drm (t,\xbfm) \right] \\
  &= \lim_{n \to \infty} \left[ \int_{Q_T} (\partial_t + \Delta) g_n u \, \drm (t,\xbfm) 
  	+ \int_{Q_T} g_n (\partial_t - \Delta) u \, \drm (t,\xbfm) \right] = 0.
\end{align*}
The expression above is equal to zero, because $u = 0$ for $t =0$, 
$g_n = 0$ for $t=T$, and $g_n$ has a zero boundary condition. 
\end{proof}

\begin{remark}\label{rem.density_C2_H1}
If we consider $t \mapsto T-t$ in Lemma \ref{lem.density_C2_H1}, we obtain that 
$\widehat{C}^2 (\overline{Q}^T)$ is dense in $H^{1, \frac12}_{;,0} (Q^T; - \partial_t - \Delta)$, 
where 
\[ 
\widehat{C}^2 (\overline{Q}_T) \isdef \Big\lbrace u \colon u \circ \kappabfm 
\in C^2_0 \big( [0,T) \times \overline{\Omega}_0 \big) \Big\rbrace
\]
and $Q^T$ is the time flipped $Q_T$. Since $Q_T$ was arbitrary, we have 
$\widehat{C}^2 (\overline{Q}_T)$ is dense in $H^{1, \frac12}_{;,0} (Q_T;  \partial_t + \Delta)$.
\end{remark}

Next, we will introduce a lemma concerning the trace maps, which will 
be later used in the proof of the jump relations. It is the analogue of 
\cite[Lemma 2.23]{Costabel1990}.

\begin{lemma}\label{lem.combined_trace_map_dense}
The combined trace map $(\gamma_0, \gamma_1^+) \colon u \mapsto 
(\gamma_0 u, \gamma_1^+ u)$ maps $\widehat{C}^2 (\overline{Q}_T)$ 
onto a dense subspace of $H^{\frac12, \frac14} (\Sigma_T) \times 
H^{-\frac12, -\frac14} (\Sigma_T)$. 
\end{lemma}

\begin{proof}
We again mimic the respective proof from \cite{Costabel1990}, 
but will not use a time reversal map. Let us assume a linear functional 
$(\chi, \psi) \in H^{\frac12, \frac14} (\Sigma_T) \times H^{-\frac12, -\frac14}(\Sigma_T)$ 
that vanishes on the range of $(\gamma_1^+, \gamma_0)$. We need to 
show that $(\chi, \psi) = (0,0)$, since then the density follows by 
\cite[Korollar III.1.9]{Werner2018}.  To this end, assume
\begin{equation}\label{eq.proof_combined_trace_map_equal}
 \langle \chi, \gamma_1^+ \varphi \rangle = \langle \psi, \gamma_0 \varphi \rangle 
 	\quad \text{for all } \varphi \in \widehat{C}^2 (\overline{Q}_T).
\end{equation}
Let
\[ 
  \Tcal = (g \mapsto \Tcal g) \colon H^{\frac12, \frac14} (\Sigma_T) \to H^{1, \frac12}_{;0,} (Q_T)
\]
be the solution operator (see Theorem \ref{thm.solution_operator_isomorphism}) 
of the Dirichlet problem
\begin{equation}
\label{eq.Dirichlet_Prob_solution_op}
\begin{aligned}
(\partial_t - \Delta) (\Tcal g) & = 0 && \ \ \text{in } Q_T, \\
 \gamma_0 (\Tcal g) & = g && \ \ \text{on } \Sigma_T.
\end{aligned}
\end{equation}
Moreover, let 
\[ 
\Scal = (f \mapsto \Scal f) \colon L^2 (Q_T) \to H^{1, \frac12}_{0;,0} (Q_T) 
\]
be the solution operator (see Theorem \ref{thm.solution_operator_isomorphism} 
used for the substitution $t \mapsto T-t$) of the Dirichlet problem
\[
\begin{aligned}
(\partial_t + \Delta) (\Scal f) & = f && \ \ \text{in } Q_T, \\
\gamma_0 (\Scal f) & = 0 && \ \ \text{on } \Sigma_T.
\end{aligned}
\]
	
We can apply Green's second formula from Lemma \ref{lem.green_formula_tube} 
to $u \isdef \Tcal\chi$ and $v \isdef \Scal f$ for any $f \in L^2 (Q_T)$, since 
$u \in H^{1, \frac12}_{;0,} (Q_T; \partial_t - \Delta)$ and $(\partial_t + \Delta) v \in L^2 (Q_T)$. 
We obtain
\[ 
  \int_{Q_T} \big\{(\partial_t - \Delta) u v +  u (\partial_t + \Delta) v \big\}\, \drm (\xbfm, t) 
  = \langle \gamma_0 u, \gamma_1^+  v \rangle - \langle \gamma_1^- u, \gamma_0 v \rangle.
\]
Since $\gamma_0  v = 0$ and $\gamma_0 u = \chi$, as well as $(\partial_t - \Delta) u = 0$ 
and $(\partial_t + \Delta) v = f$, we obtain
\[  
  \int_{Q_T} u  f \, \drm (t,\xbfm) =  \langle  \chi, \gamma_1^+  v \rangle. 
\]
	
Due to continuity and Remark \ref{rem.density_C2_H1}, \eqref{eq.proof_combined_trace_map_equal} 
holds also for all $\varphi \in H^{1, \frac12}_{;,0} (Q_T; \partial_t + \Delta)$ and, thus, 
also for $\varphi = \Scal f$. This implies
\[ 
  \langle \chi, \gamma_1^+ \Scal f \rangle = \langle \psi, \gamma_0 \Scal f \rangle.
\]
Since $\gamma_0 \Scal f = 0$, we thus obtain $\int_{Q_T} u  f \, \drm(t,\xbfm) = 0$ 
for all $f \in L^2 (Q_T)$. Therefore, $0 = u = \Tcal ( \chi)$ and thus $\chi = \gamma_0 u = 0$. 
Looking again at \eqref{eq.proof_combined_trace_map_equal} gives
\[ 
\langle \psi, \gamma_0 \varphi \rangle = 0 \quad \text{for all } \varphi \in H^{1, \frac12}_{;,0} (Q_T).
\]
The trace map $\gamma_0$ is not only surjective for $\varphi \in H_{;0,}^{1, \frac12}(Q_T)$ 
as shown in Section \ref{subsec.trace_op_non_cylindrical_domains}, but also 
for $\varphi \in H_{;,0}^{1, \frac12} (Q_T)$ if one considers the backward 
problem. We may hence conclude that $\psi = 0$.
\end{proof}

In the following, we state the analogue of \cite[Proposition 2.24]{Costabel1990}. 
\begin{lemma}\label{lem.green_formula_for_and_backward}
Green's first formula given in \eqref{eq.first_green_formula} holds 
for all $u \in H^{1, \frac12}_{;0,} (Q_T; \partial_t - \Delta)$ and $v \in 
H^{1, \frac12}_{;0,} (Q_T)$. If also $v \in H^{1, \frac12}_{;0,} 
(Q_T; \partial_t - \Delta)$, we can write the Green's formula as
\begin{equation}\label{eq.rewritten_green_formula}
\begin{aligned}
  &\int_{Q_T} \langle \nabla u, \nabla v \rangle \, \drm (t,\xbfm) - d(v,u) + \int_{\Omega_T} u(T, \xbfm) v(T, \xbfm) \, \drm \xbfm \\
  &\hspace*{4cm} = \langle \gamma_1^- u, \gamma_0 v \rangle + \int_{Q_T} (\partial_t - \Delta) u v \, \drm (t,\xbfm). 
\end{aligned}
\end{equation}
\end{lemma}

\begin{proof}
We again mimic the proof of \cite[Proposition 2.24]{Costabel1990}. Given 
$u \in \widetilde{C}^2 (\overline{Q}_T)$ and $v \in \widetilde{C}^1(\overline{Q}_T)$%
\footnote{The space $\widetilde{C}^1(\overline{Q}_T)$ ist defined in 
complete analogy to $\widetilde{C}^2 (\overline{Q}_T)$ via
\[
\widetilde{C}^1 (Q_T) \isdef \big\lbrace u \colon u \circ \kappabfm 
\in C_0^1 \big( (0,T] \times \Omega_0 \big) \big\rbrace
\]}, 
we find
\begin{equation}\label{eq.green_formula_in_proof}
\begin{aligned}
  &\int_{Q_T} \langle \nabla u, \nabla v \rangle \, \drm (t,\xbfm)
   + \int_{Q_T} \partial_t u v \, \drm (t,\xbfm) 
   + \frac12 \int_0^T \int_{\Gamma_t} \langle \Vbfm, \nbfm \rangle u v \drm \sigma \drm t \\
  &\hspace*{5.5cm} = \langle \gamma_1^- u, \gamma_0 v \rangle + \int_{Q_T} (\partial_t - \Delta) u v \, \drm (t,\xbfm).
\end{aligned}
\end{equation}
All terms are continuous with respect to $v$ in the $H^{1,\frac12}_{;0,} (Q_T)$-norm. 
Thus, by continuity, we can extend \eqref{eq.green_formula_in_proof} to all
$v \in H^{1, \frac12}_{;0,} (Q_T)$. Let $v \in H^{1, \frac12}_{;0,} (Q_T)$ 
be fixed. Then, all terms in \eqref{eq.green_formula_in_proof} except 
the term containing the $\partial_t u$ are obviously continuous with 
respect to $u$ in the norm of $H^{1, \frac12}_{;0,} (Q; \partial_t- \Delta)$. 
Therefore, also the term containing $\partial_t u$ is continuous. 
Lemma \ref{lem.density_C2_H1} allows to extend 
\eqref{eq.green_formula_in_proof} to all $u \in H^{1, \frac12}_{;0,} 
(Q_T; \partial_t - \Delta)$. Thus, Green's first formula holds as given in the claim.
	
For $u, v \in \widetilde{C}^2 (\overline{Q}_T)$, \eqref{eq.rewritten_green_formula} 
holds. As in \cite{Costabel1990}, the term $\int_{\Omega_T} u(T, \xbfm) v(T, \xbfm) \, 
\drm(t,\xbfm)$ is continuous for $u$ and $v$ in the norm of $H^{1,\frac12}_{;0,} 
(Q_T; \partial_t - \Delta)  \subset \Vcal (Q_T)$ and $\Vcal (Q_T)$ consists of
functions $\varphi$, which always satisfy $\varphi \circ \kappabfm \in 
C \big( [0,T]; L^2 (\Omega_0) \big)$. From here, the second claim follows.
\end{proof}

\section{The Calder{\'o}n operator}
\label{sec.calderon_op}
Let us first introduce the fundamental 
solution for the heat equation, which, 
in accordance with e.g.\ \cite{Tausch2019}, reads
\[ 
G (t, \tau, \xbfm, \ybfm) \isdef \begin{cases} 
\frac{1}{(4\pi (t-\tau))^{\frac{d}{2}}} \exp \left(- \frac{ \| \xbfm - \ybfm \|^2}{4 (t-\tau )}\right),\ \ 
& \text{ if } \tau < t,  \\ 0, \ \ & \text{ if } \tau \geq t,\end{cases} \]
Notice that this is equivalent to consider $\overline{G} (t-\tau, \xbfm, \ybfm)$, 
where $\overline{G}$ is given by
\[ 
\overline{G} (\tau, \xbfm, \ybfm)\isdef \frac{1}{(4 \pi \tau)^{\frac{d}{2}}} 
\exp\left( -\frac{ \| \xbfm - \ybfm \|^2}{4 \tau} \right) \frac12 (1 + \operatorname{sign} \tau),
\]
as introduced in \cite[Formula (2.39)]{Costabel1990}.
Moreover, let us denote 
\begin{equation}\label{eq.tilde_G}
   \widetilde{G}(t, \xbfm) = \begin{cases} \frac{1}{(4\pi t)^{\frac{d}{2}}} 
   \exp\left(-\frac{\|\xbfm\|^2}{4 t}\right),\ \ & \text{ if } t > 0,  \\ 0, \ \ & \text{ if } t \leq 0. \end{cases} 
\end{equation}

For $u \in C^2\big(\bigcup_{0 < t< \infty}(\lbrace t \rbrace \times \overline{\Omega}_t)\big)$ 
with $u(0, \xbfm) = 0$ on $\Omega_0$, we have for $(t_0,\xbfm_0) 
\in \bigcup_{0 < t< \infty}( \lbrace t \rbrace \times \Omega_t)$ that
\begin{align*}
u (t_0, \xbfm_0) &= \int_{Q_T} \int_{\Omega_t} (\partial_t - \Delta) u (t, \xbfm) 
\widetilde{G}(t_0 - t, \xbfm_0 - \xbfm) \, \drm(t,\xbfm)\\
  &\quad +  \int_{\Sigma_T} \widetilde{G}( t_0 - t, \xbfm_0 - \xbfm) 
  	\frac{\partial u}{\partial \nbfm} (t, \xbfm) \,\drm\sigma_{(t,\xbfm)} \\
  & \quad - \int_{\Sigma_T} \frac{\partial \widetilde{G}}{\partial \nbfm} 
  	(t_0 - t, \xbfm_0 - \xbfm) u (t, \xbfm) \, \drm\sigma_{(t,\xbfm)} \\
  & \quad + \int_{\Sigma_T} \widetilde{G}( t_0 - t, \xbfm_0 - \xbfm) 
    \langle \Vbfm, \nbfm \rangle  (t, \xbfm) u (\ybfm, \tau) \drm\sigma_{(t,\xbfm)},
\end{align*}
as it can be seen from Lemma \ref{lem.green_formula_tube} and 
the property of the fundamental solution. Moreover, we will only 
look at the case, for which $(\partial_t - \Delta) u = 0$ holds.

We introduce the single and double layer potentials as
\begin{align*}
\widetilde{\Vcal} \varphi (t_0, \xbfm_0) &\isdef \langle \varphi, \gamma_0 G \rangle 
=  \int_{\Sigma_T} G (t_0, t, \xbfm_0, \ybfm) \varphi (t, \ybfm) \, \drm\sigma_{(t,\ybfm)}, \\
\widetilde{\Kcal} w (t_0, \xbfm_0) &\isdef \langle \gamma_1^+ G, w \rangle 
= \int_{\Sigma_T} \gamma_{1, (t, \ybfm)}^+ G (t_0, t, \xbfm_0, \ybfm) w(t, \ybfm) \, \drm\sigma_{(t,\ybfm)}.
\end{align*}
Then, similarly to \cite[Theorem 2.20]{Costabel1990}, we obtain 
the representation formula from \cite[Equation (6)]{Tausch2019}
given in the following lemma.

\begin{lemma}
Let $u \in H^{1, \frac12} (Q_T)$ with $(\partial_t - \Delta) u = 0$ in $Q_T$. 
Then, we have the representation formula
\begin{equation}\label{eq.representation_formula}
  u (t, \widetilde{\xbfm}) = \widetilde{\Vcal} \gamma_1^- u (t, \widetilde{\xbfm}) 
  	- \widetilde{\Kcal} u (t, \widetilde{\xbfm}) \quad \text{for all } (t, \widetilde{\xbfm}) \in Q_T. 
\end{equation}
\end{lemma}

As in \cite[pg.~514]{Costabel1990}, we can rewrite the 
definition of the single layer potential by 
\begin{equation}\label{eq.V_tilde_with_G_tilde}
\begin{aligned}
  \widetilde{\Vcal} \varphi (t_0, \xbfm_0) & = \langle \varphi, \gamma_0 G(t, t_0, \xbfm, \xbfm_0) \rangle \\
  &= \langle \gamma_0 '  \varphi, G(t, t_0, \xbfm, \xbfm_0) \rangle \\
  & = \widetilde{G} \star (\gamma_0 ' \varphi) (t_0, \xbfm_0),
\end{aligned}
\end{equation}
where $\widetilde{G}$ is given in \eqref{eq.tilde_G} and
\[ 
\langle \gamma_0 ' \varphi, \chi \rangle 
= \langle \varphi, \gamma_0 \chi \rangle 
= \int_{\Sigma_T} \varphi \chi \, \drm \sigma_{(t, \xbfm)}
\]
for all $\chi \in C_0^\infty (\Rbbb^{d+1})$. We will 
use this also for $\chi \in C^2_0 (\Rbbb^{d+1})$. 

We would like to find the mapping properties of the single and double 
layer potentials, which are the equivalent of the results given in 
\cite[Proposition 3.1, Remark 3.2, and Proposition 3.3]{Costabel1990}.

\begin{lemma}\label{lem.mapping_properties_single_layer_pot}
The mapping
\[ 
  \widetilde{\Vcal} \colon H^{-\frac12,-\frac14} (\Sigma_T) \to H^{1, \frac12}_{;0,} (Q_T; \partial_t - \Delta) 
\]
is continuous.
\end{lemma}

\begin{proof}
The proof follows as in \cite[pg.~514--515]{Costabel1990} in the 
case of a cylindrical domain. In there, the claim is proven by considering 
the problem on $\Rbbb^{d+1}$ using Fourier techniques and then restricting 
it appropriately, which can also be done in the case of a non-cylindrical domain.
\end{proof}

\begin{lemma}\label{lem.mapping_properties_double_layer_pot}
The mapping
\[ 
\widetilde{\Kcal} \colon H^{\frac12, \frac14} (\Sigma_T)
\to H^{1, \frac12}_{;0,} (Q_T; \partial_t - \Delta)
\]
is continuous.
\end{lemma}

\begin{proof}
The proof is in complete analogy to \cite[pg.~515]{Costabel1990}, 
but we repeat it for the convenience of the reader. We consider 
the solution operator $\Tcal$, which maps the Dirichlet data $g$ 
to the solution $u \isdef \Tcal g$ of the partial differential equation 
\eqref{eq.Dirichlet_Prob_solution_op}.
According to Theorem \ref{thm.solution_operator_isomorphism}, 
the solution operator $\Tcal$ is a continuous mapping
\begin{equation}\label{eq.mapping_property_sol_op}
\Tcal \colon  H^{\frac12, \frac14} (\Sigma_T) 
\to H^{1, \frac12}_{;0,} (Q_T; \partial_t - \Delta). 
\end{equation}
The representation formula \eqref{eq.representation_formula} 
yields $u (t, \widetilde{\xbfm}) = \widetilde{\Vcal} \gamma_1^- 
u (t, \widetilde{\xbfm}) - \widetilde{\Kcal} u (t, \widetilde{\xbfm})$ 
and thus $\Tcal g = \widetilde{\Vcal} \gamma_1^- \Tcal g - \widetilde{\Kcal} g$ 
for all $g \in H^{\frac12, \frac14} (\Sigma_T)$. Rearranging gives hence
\[ 
\widetilde{\Kcal} = \widetilde{\Vcal} \gamma_1^- \Tcal - \Tcal. 
\]
The claim follows now by using the mapping property \eqref{eq.mapping_property_sol_op}
of $\Tcal$, $\widetilde{\Vcal}$ (Lemma \ref{lem.mapping_properties_single_layer_pot}), and 
$\gamma_1^-$ (Lemma \ref{lem.mapping_properties_gamma_1_minus}).
\end{proof}

We can take the traces $\gamma_0$ of the single and double layer potential. 
Let the radius $R$ be large enough such that the boundary $\Gamma_t$ is contained 
in the ball $B_R \isdef \big\lbrace \xbfm \in \Rbbb^d \colon \| \xbfm \| < R \big\rbrace$ 
and set $\Omega_t^c \isdef B_R \backslash \overline{\Omega}_t$ and 
$Q_T^c \isdef\bigcup_{-\infty < t< T}(\lbrace t \rbrace \times \Omega_t^c)$. 
Lemmata \ref{lem.mapping_properties_single_layer_pot} and 
\ref{lem.mapping_properties_double_layer_pot} provide 
also the continuity of the mappings
\[ 
\widetilde{\Vcal} \colon H^{-\frac12, -\frac14} (\Sigma_T) 
\to H^{1, \frac12}_{;0,} (Q_T^c; \partial_t - \Delta)
\]
and
\[ 
\widetilde{\Kcal} \colon H^{\frac12, \frac14} (\Sigma_T) 
\to H^{1, \frac12}_{;0,} (Q_T^c; \partial_t - \Delta). 
\]

In order to state the tube analogue of \cite[Theorem 3.4]{Costabel1990},
we define the jumps as in \cite[Formula (3.16)]{Costabel1990} in accordance 
with
\[
[\gamma_0 u] \isdef \gamma_0(u|_{Q_T^c}) - \gamma_0 (u|_{Q_T}), \quad
[\gamma_1^\pm u ] \isdef \gamma_1^\pm (u|_{Q_T^c}) - \gamma_1^\pm (u|_{Q_T}).
\]
We then have:

\begin{lemma}\label{lem.jump_relations}
For all $\psi \in H^{-\frac12, -\frac14} (\Sigma_T)$ and all 
$w \in H^{\frac12, \frac14} (\Sigma_T)$, there hold the 
jump relations
\[
\begin{aligned}
  [\gamma_0 \widetilde{\Vcal} \psi ] &= 0, \ \ && [\gamma_1^- \widetilde{\Vcal} \psi ] = - \psi, \\
  [\gamma_0 \widetilde{\Kcal} w ] & = w, \ \ && [\gamma_1^- \widetilde{\Kcal} w]  = 0.
\end{aligned}
\]
\end{lemma}

\begin{proof}
We mimic the proof of \cite{Costabel1990} without using the
time reversal map. Let $\psi \in H^{-\frac12, -\frac14} (\Sigma_T)$. 
We set $u \isdef \widetilde{\Vcal}\psi$. Due to the mapping property 
of the single layer potential, we then have $ u \in H^{1, \frac12}_{;0,} 
\big( (0,T) \times B_R (0) \big)$ and thus, by the trace lemma, we 
have $\gamma_0 (u|_{Q_T}) = \gamma_0 (u|_{Q_T^c})$. 
	
Let us next consider the normal jump of $\widetilde{\Vcal}$. 
From \eqref{eq.V_tilde_with_G_tilde}, we obtain by considering 
$u = \widetilde{\Vcal} \psi$
\[ 
(\partial_t - \Delta) u = \gamma_0 ' \psi
\]
in $\Rbbb_+ \times \Rbbb^d$. We consider any test function 
$\varphi \in C^2_0 \big( (0,T) \times B_R \big)$ and we obtain
\[ 
\langle \psi, \gamma_0 \varphi \rangle = \langle \gamma_0' \psi, \varphi \rangle 
= \big\langle (\partial_t - \Delta) u , \varphi \big\rangle 
= - \big\langle u, (\partial_t + \Delta)  \varphi \big\rangle, 
\]
where the last equality holds due to the integration by parts 
on a cylindrical domain. We thus have
\begin{equation}\label{eq.proof_jump_cond_representation_1}
\langle \psi, \gamma_0  \varphi \rangle 
= - \int_{(0,T) \times B_R} (\partial_t + \Delta) \varphi  u \, \drm (t,\xbfm).
\end{equation}

On the other hand, we can use Green's second formula, 
given in Lemma \ref{lem.green_formula_tube} in $Q_T$ and $Q_T^c$, 
where we use that $(\partial_t - \Delta) u = 0$ in $Q_T \cup Q_T^c$. 
This yields
\[ \int_{Q_T} (\partial_t + \Delta) \varphi  u \, \drm (t,\xbfm) 
=  \langle \gamma_0 u, \gamma_1^+  \varphi \rangle - \langle \gamma_1^- u , \gamma_0  \varphi \rangle 
\]
and
\[ \int_{Q_T^c} (\partial_t + \Delta) \varphi \ u \, \drm (t,\xbfm) 
= -  \langle \gamma_0 u, \gamma_1^+  \varphi \rangle +  \langle \gamma_1^- u, \gamma_0  \varphi \rangle.
\]
Adding these two expressions yields
\begin{equation}\label{eq.proof_jump_cond_representation_2}
\int_{(0,T) \times B_R} (\partial_t + \Delta) \varphi  u \, \drm (t,\xbfm) 
=  \big\langle [ \gamma_1^- u ], \gamma_0  \varphi \rangle,
\end{equation}
where we used $[\gamma_0 u] = 0 = [\gamma_0  \varphi ] = [ \gamma_1^+ \varphi]$. 
Comparing \eqref{eq.proof_jump_cond_representation_1} with
\eqref{eq.proof_jump_cond_representation_2} results in $[\gamma_1^- u] = - \psi$.
	
We are left with proving the jump relations for the double layer 
potential. To that end, we choose $w \in H^{\frac12, \frac14} (\Sigma_T)$ 
and define $u \isdef \widetilde{\Kcal} w$. Let $\varphi \in C_0^\infty(\Rbbb_+ \times B_R)$ 
be a test function. As above, we obtain
\begin{equation}\label{eq.proof_jump_cond_representation_3}
\int_{(0,T) \times B_R} (\partial_t + \Delta) \varphi  u \, \drm (t,\xbfm) 
= \big\langle [\gamma_1^- u], \gamma_0  \varphi \big\rangle 
- \big\langle [\gamma_0 u], \gamma_1^+  \varphi \big\rangle. 
\end{equation}
For $\widetilde{\Kcal}$, we obtain $\widetilde{\Kcal} w 
= \widetilde{G} \star \big( (\gamma_1^+)' w \big)$ similar to 
\eqref{eq.V_tilde_with_G_tilde}. Therefore, we have $(\partial_t - \Delta) 
\widetilde{\Kcal} w = (\gamma_1^+)' w$ in $\Rbbb_+ \times B_R$. 
From here, it follows that
\begin{equation}\label{eq.proof_jump_cond_representation_4}
  - \int_{(0,T) \times B_R} (\partial_t + \Delta) \varphi  u \, \drm (t,\xbfm)
  	= \big\langle (\partial_t - \Delta) u,  \varphi \big\rangle
	= \big\langle (\gamma_1^+)' w,  \varphi \big\rangle
	= \langle w, \gamma_1^+  \varphi \rangle.
\end{equation}
Comparing \eqref{eq.proof_jump_cond_representation_3} with
\eqref{eq.proof_jump_cond_representation_4} yields
\begin{equation}\label{eq.proof_jump_cond_representation_5}
\big\langle [\gamma_1^- u], \gamma_0  \varphi \big\rangle 
= \big\langle [\gamma_0 u] - w, \gamma_1^+  \varphi \big\rangle
\end{equation}
for all $\varphi \in C_0^2(\Rbbb_+ \times B_R)$. 
Applying Lemma \ref{lem.combined_trace_map_dense} says 
that both sides of \eqref{eq.proof_jump_cond_representation_5} 
have to vanish identically, from where $[\gamma_1^- u] =0$
and $[\gamma_0 u] = w$ follows.
\end{proof}

Now, as in \cite[Definition 3.5]{Costabel1990}, we are in the
position to define the boundary integral operators. 

\begin{definition}\label{def.layer_operators}
Let $\psi \in H^{-\frac12, -\frac14} (\Sigma_T)$ and 
$w\in H^{\frac12, \frac14} (\Sigma_T)$. We can then 
define the single layer operator as
\[ 
\Vcal \psi \isdef \gamma_0 \widetilde{\Vcal} \psi,
\]
the adjoint double layer operator as
\[ 
\Kcal' \psi \isdef \frac12 \left( \gamma_1^- (\widetilde{\Vcal} \psi )|_{Q_T} 
+ \gamma_1^- (\widetilde{\Vcal} \psi)|_{Q_T^c} \right),  
\]
the double layer operator as
\[ \Kcal w \isdef \frac12 \left( \gamma_0 (\widetilde{\Kcal} w)|_{Q_T} 
+ \gamma_0(\widetilde{\Kcal} w)|_{Q_T^c} \right) 
\]
and the hypersingular operator as
\[ 
\Dcal w \isdef - \gamma_1^- \widetilde{\Kcal} w.
\]
\end{definition}

As in \cite[Theorem 3.7]{Costabel1990}, we have the following 
mapping properties of these operators.

\begin{theorem}
The boundary integral operators from Definition 
\ref{def.layer_operators} are continuous mappings as follows
\begin{align*}
  \Vcal \colon  H^{-\frac12, -\frac14} (\Sigma_T) &\to H^{\frac12, \frac14} (\Sigma_T), \\
  \Kcal ' \colon H^{-\frac12, -\frac14} (\Sigma_T) &\to H^{-\frac12, -\frac14} (\Sigma_T), \\
  \Kcal \colon  H^{\frac12, \frac14} (\Sigma_T) &\to H^{\frac12, \frac14} (\Sigma_T), \\
  \Dcal \colon  H^{\frac12, \frac14} (\Sigma_T) & \to H^{-\frac12, -\frac14} (\Sigma_T).
\end{align*}
\end{theorem}

\begin{proof}
The assertion follows immediately by using the mapping 
properties of the layer potentials from Lemma 
\ref{lem.mapping_properties_single_layer_pot} and 
Lemma \ref{lem.mapping_properties_double_layer_pot} 
as well as of the trace operators introduced in Section 
\ref{subsec.trace_op_non_cylindrical_domains} and 
from Lemma \ref{lem.mapping_properties_gamma_1_minus}. 
\end{proof}

We can state the analogue of \cite[Formulae (3.24)--(3.27)]{Costabel1990}.
\begin{lemma}\label{lem.jump_relations_one_sided}
It holds
\begin{align*}
  \gamma_0 (\widetilde{\Vcal} \psi) |_{Q_T} & = \gamma_0 (\widetilde{\Vcal} \psi)|_{Q_T^c} = \Vcal \psi, \\
  \gamma_1^- (\widetilde{\Vcal} \psi)|_{Q_T} & = \frac12 \psi + \Kcal ' \psi, \\
  \gamma_1^- (\widetilde{\Vcal} \psi)|_{Q_T^c} & = - \frac12 \psi + \Kcal ' \psi, \\
  \gamma_0 (\widetilde{\Kcal} w)|_{Q_T} & = - \frac12 w + \Kcal w, \\
  \gamma_0 (\widetilde{\Kcal} w)|_{Q_T^c} & = \frac12 w + \Kcal w, \\
  \gamma_1^- (\widetilde{\Kcal}  w)|_{Q_T} & = \gamma_1^- (\Kcal w)|_{Q_T^c} = - \Dcal w.
\end{align*}
\end{lemma}

\begin{proof}
We just prove the second statement, as the other statements follow 
similarly. According to Lemma \ref{lem.jump_relations}, we have
\[ 
[\gamma_1^- \widetilde{\Vcal} \psi] 
= \gamma_1^- \widetilde{\Vcal} \psi|_{Q_T^c} - \gamma_1^- \widetilde{\Vcal} \psi|_{Q_T} 
= - \psi. 
\]
Therefore,
\[ \gamma_1^- \widetilde{\Vcal} \psi|_{Q_T} 
= \psi + \gamma_1^- \widetilde{\Vcal} \psi|_{Q_T^c}. 
\]
By Definition \ref{def.layer_operators}, we have
\[ \Kcal ' \psi \isdef \frac12 \left( \gamma_1^- \widetilde{\Vcal} \psi |_{Q_T} 
+ \gamma_1^- \widetilde{\Vcal} \psi|_{Q_T^c} \right). 
\]
Substituting this into the expression above yields
\[ \gamma_1^- \widetilde{\Vcal} \psi|_{Q_T} 
= \psi + 2 \Kcal ' \psi - \gamma_1^- \widetilde{\Vcal} \psi|_{Q_T}, 
\]
from where the claim follows immediately. 
\end{proof}

\begin{remark}
Following \cite[Formulae (7)--(10)]{Tausch2019}, the relations 
in the interior given in Lemma \ref{lem.jump_relations_one_sided} 
can also be written as
\begin{align*}
\gamma_0 \widetilde{\Vcal} \psi (t, \xbfm) 
&= \int_{\Sigma_T} G(t, \tau, \xbfm, \ybfm) \psi (\tau, \ybfm) \, \drm\sigma_{(\tau,\ybfm)},\\
\gamma_1^-  \widetilde{\Vcal} \psi (t, \xbfm) 
&= \frac12 \psi (t, \xbfm) + \int_{\Sigma_T} \gamma_{1, (t,\xbfm)}^- 
G(t, \tau, \xbfm, \ybfm) \psi(\tau, \ybfm) \, \drm\sigma_{(\tau,\ybfm)},\\
\gamma_0 \widetilde{\Kcal} w (t, \xbfm) 
&= -\frac12 w(t, \xbfm) + \int_{\Sigma_T} \gamma_{1, (\tau, \ybfm)}^+ 
G(t, \tau, \xbfm, \ybfm) w (\tau, \ybfm) \, \drm\sigma_{(\tau,\ybfm)}, \\
\gamma_1^- \widetilde{\Kcal} w (t, \xbfm) 
&= - \int_{\Sigma_T} \gamma_{1, (t, \xbfm)}^- \gamma_{1, (\tau, \ybfm)}^+ 
G(t, \tau, \xbfm, \ybfm) w (\tau, \ybfm) \, \drm\sigma_{(\tau,\ybfm)}.
\end{align*}
\end{remark}

We can take the traces in the representation formula \eqref{eq.representation_formula} 
to obtain the Dirichlet data and the Neumann data of the solution $u$ 
of the homogeneous heat equation. This yields
\begin{align}\label{eq.boundary_integral_eq1}
  \gamma_0 u & = \frac12 \gamma_0 u - \Kcal \gamma_0 u + \Vcal \gamma_1^- u, \\
\label{eq.boundary_integral_eq2}
  \gamma_1^- u & = \Dcal \gamma_0 u + \frac12 \gamma_1^- u + \Kcal ' \gamma_1^- u,
\end{align}
compare also \cite[Formulae (11) and (12)]{Tausch2019}.

As in \cite[pg.~518]{Costabel1990}, we can define the 
Calder{\'o}n projector and the associated involution $\Acal$ as
\[
  \Ccal_{Q_T} \isdef \frac12 \on{id} + \Acal \isdef \frac12 \on{id} + \begin{bmatrix}
	- \Kcal & \Vcal \\
	\Dcal & \Kcal ' \end{bmatrix}.
\]
We state next the analogue of \cite[Theorem 3.9]{Costabel1990}.

\begin{theorem}
The operator $\Ccal_{Q_T}$  is a projection operator in the space
\[ 
  \Hcal \isdef H^{\frac12, \frac14} (\Sigma_T) \times H^{-\frac12, -\frac14} (\Sigma_T).
\]
The following statements are equivalent for $(w, \psi) \in \Hcal$:
\begin{enumerate}[(i)]
\item There is a $u \in H^{1, \frac12}_{;0,} (Q_T)$ with $(\partial_t - \Delta ) u = 0$ 
in $Q_T$ and $w = \gamma_0 u$, $\psi = \gamma_1^- u$ on $\Sigma_T$. 
\item It holds
\[ 
\begin{bmatrix} w \\ \psi \end{bmatrix} = \Ccal_{Q_T} 
	\begin{bmatrix}w \\ \psi \end{bmatrix}.
\]
\end{enumerate}
\end{theorem}

\begin{proof}
We again follow the proof of \cite[Thereom 3.9]{Costabel1990}.
	
(i) $\Rightarrow $ (ii) follows by the considerations above, especially 
in \eqref{eq.boundary_integral_eq1} and \eqref{eq.boundary_integral_eq2}.
	
For the proof of (ii) $\Rightarrow$ (i), $\psi$ and $w$ are given and we define
\begin{equation}\label{eq.proof_calderon_op_def_u}
  u \isdef \widetilde{\Vcal} \psi - \widetilde{\Kcal} w.
\end{equation}
Using the mapping properties of the potentials implies 
that $u \in H^{1, \frac12}_{;0,} (Q_T)$ and we obtain
\begin{equation}\label{eq.proof_calderon_op_gamma_u_and_w}
  \begin{bmatrix}\gamma_0 u \\ \gamma_1^- u\end{bmatrix} = \Ccal_{Q_T} 
   \begin{bmatrix} w \\ \psi \end{bmatrix}.
\end{equation}
Since the right hand side equals to $[w, \psi]^{\intercal}$ 
according to (ii), the claim follows immediately.
	
It remains to show the projection property of $\Ccal_{Q_T}$. Since on
the one hand $[\gamma_0 u, \gamma_1^- u]^{\intercal} = \Ccal_{Q_T} 
[\gamma_0 u, \gamma_1^- u]^{\intercal}$ holds according to 
\eqref{eq.boundary_integral_eq1}, \eqref{eq.boundary_integral_eq2}
and on the other hand \eqref{eq.proof_calderon_op_gamma_u_and_w} 
holds for any $[w, \psi]^{\intercal}$ and $u$ given by \eqref{eq.proof_calderon_op_def_u}, 
we obtain $\Ccal_{Q_T} [w, \psi]^{\intercal} = \Ccal_{Q_T}^2 [w, \psi]^{\intercal}$ 
for any $[w, \psi]^{\intercal}$ and thus
\begin{equation}\label{eq.projection_property_calderon}
   \Ccal_{Q_T}^2 = \Ccal_{Q_T}.
\end{equation}
\end{proof}

We can state the following corollary in analogy to \cite[Corollary 3.10]{Costabel1990}.

\begin{corollary}
The operator $\Acal \colon \Hcal \to \Hcal$ is an isomorphism.
\end{corollary}

\begin{proof}
We use the same argument as in the proof of \cite[Corollary 3.10]{Costabel1990}.
Notice that we can reformulate \eqref{eq.projection_property_calderon} as follows:
\[ 
  \left(\frac12 \on{id} + \Acal\right)^2 = \frac14 \on{id} + \Acal + \Acal^2 = \frac12 \on{id} + \Acal.
\]
We hence conclude
\begin{equation}\label{eq.projection_property_calderon_rewritten}
  \Acal^2 = \frac14 \on{id},
\end{equation}
which is equivalent to
\[ 
  \Acal^{-1} = 4 \Acal. 
\]
\end{proof}

As in \cite{Costabel1990}, we can interchange the columns of the 
operator $\Acal$ to define the operator
\[ 
 A \isdef \begin{bmatrix} \Vcal & - \Kcal \\ \Kcal' & \Dcal \end{bmatrix},
 \]
which is an isomorphism of the space
\[ 
  \Hcal' \isdef H^{-\frac12, -\frac14} (\Sigma_T) \times H^{\frac12, \frac14} (\Sigma_T)
\]
onto its dual space $\Hcal$. Following \cite{Costabel1990}, we define 
the duality product between $\Hcal'$ and $\Hcal$ in accordance with
\[
  \left\langle \begin{bmatrix} \psi \\ w \end{bmatrix}, 
  \begin{bmatrix} v \\ \varphi \end{bmatrix} \right\rangle 
  	\isdef \langle \psi, v \rangle + \langle \varphi, w \rangle
\]
for all $v$, $w \in H^{\frac12, \frac14}(\Sigma_T)$ and $\varphi$, 
$\psi \in H^{-\frac12, - \frac14} (\Sigma_T)$. We are now in the position to
state the analogue of \cite[Theorem 3.11]{Costabel1990}, which is the positive 
definiteness of the operator $A$. 

\begin{theorem}\label{thm.positive_definiteness_calderon}
There exists a constant $\alpha > 0$ such that
\[
  \left\langle \begin{bmatrix} \psi \\ w \end{bmatrix}, A 
  \begin{bmatrix} \psi \\ w \end{bmatrix} \right\rangle 
  \geq \alpha \left( \| \psi \|_{H^{-\frac12, -\frac14} (\Sigma_T)}^2 
  + \| w \|_{H^{\frac12, \frac14} (\Sigma_T)}^2 \right)
\]
for all $[\psi, w]^{\intercal} \in \Hcal '$. 
\end{theorem}

For the proof, we again mimic the proof of \cite[Theorem 3.11]{Costabel1990},
which is based on the following lemma (see \cite[Lemma 3.12]{Costabel1990}). 
Its proof can be found in \cite{Costabel1990}.

\begin{lemma}\label{lem.proof_positive_definite_help_lemma}
Let $A \colon X \to X'$ be a bounded linear operator, 
where $X'$ is the dual space of the Hilbert space $X$. 
With a compact operator $T \colon X \to X'$ and a constant $\alpha$, let $A$ satisfy
\[
  \big\langle (A+T) x, x \big\rangle \geq \alpha \|x \|_X^2 \quad \text{for all } x \in X
\]
and 
\begin{equation}\label{eq.def_of_duality_product_on_product_space}
  \langle A x, x \rangle > 0 \quad \text{for all } x \in X \backslash \lbrace 0 \rbrace.
\end{equation}
Then, there exists a constant $\alpha_1 > 0$ such that
\[ 
  \langle A x, x \rangle \geq \alpha_1 \| x \|_X^2 \quad \text{for all } x \in X.
\]
\end{lemma}

Moreover, we need the following analogue of \cite[Lemma 2.15]{Costabel1990}.

\begin{lemma}\label{lem.norm_equivalence_H_V}
Let $u \in \Vcal (Q_T)$ such that $(\partial_t - \Delta) u = 0$ in $Q_T$. 
Then, there exist constants $m_1$, $m_2$, and $m_3$ such that
\[ 
  \| u \|_{H^{1, 0} (Q_T)} \leq m_1 \| u\|_{H^{1, \frac12} (Q_T)} 
  \leq m_2 \| u \|_{\Vcal (Q_T)} \leq m_3 \| u \|_{H^{1, 0}(Q_T)}.
\]
In other words, for functions $u \in \Vcal (Q_T)$ satisfying the homogeneous 
heat equation, we have the equivalence of the norms in $\Vcal (Q_T)$,
$H^{1, 0} (Q_T)$, and $H^{1,\frac12} (Q_T)$.
\end{lemma}

\begin{proof}
The first and second inequality follow directly from the equivalence 
of norms on $Q_T$ and $Q_0$, since the proof of \cite{Costabel1990} 
is based on the definition of the norm for the first inequality and the 
interpolation result \eqref{eq.interpol_result} for the second inequality.
Nonetheless, we cannot apply the equivalence of norms on $Q_T$ 
and norms on $Q_0$ directly for the third inequality, because we have 
$(\partial_t - \Delta)u = 0$ as an assumption, which is needed to show 
the third inequality. Mapping this differential operator from the tube onto 
the cylinder oder vice versa will alter it. Therefore, we use the ideas of 
the proof of  \cite[Lemma 2.15]{Costabel1990}, but adapt them to our 
context.

Transforming the partial differential equation $(\partial_t - \Delta) u = 0$ 
from $Q_T$ back to $Q_0$ via the weak formulation (see \cite{Bruegger2020}) 
yields	
\[
   \partial_t (u \circ \kappabfm) -\Mcal (u \circ \kappabfm)  = 0\ \ \text{in } Q_0,
\]
where $\Mcal$ is defined as
\begin{align*}
  \Mcal (u \circ \kappabfm) \isdef 	
  &\on{div} \Big( (\D\kappabfm)^{-1} (\D\kappabfm)^{-\intercal} \nabla (u \circ \kappabfm) \Big) 
  + (\D\kappabfm)^{-\intercal} \nabla (u \circ \kappabfm) \cdot \partial_t \kappabfm\\
  &\hspace*{2cm} + \frac{1}{\det (\D\kappabfm)} \nabla \big(\det (\D\kappabfm) \big) 
  \cdot (\D\kappabfm)^{-1} (\D\kappabfm)^{- \intercal} \nabla (u \circ \kappabfm).
\end{align*}
By the standard theory, for fixed $t \in (0,T)$, we have that $\Mcal \colon H^1 (\Omega_0)
\to H^{-1}(\Omega_0)$ is bounded. Thus, for $u \in H^{1,0} (Q_0)$, we 
obtain $\Mcal u\in H^{-1, 0}(Q_0)\isdef [H^{1,0} (Q_0)]'$ and we conclude
\begin{align*}
  \| u \|_{\Vcal(Q_T)}^2 = \| u \circ \kappabfm \|_{\Vcal(Q_0)}^2
  &=  \| u\circ \kappabfm  \|_{H^{1,0}(Q_0)}^2 
  + \big\| \partial_t (u\circ \kappabfm) \big\|_{H^{-1,0}(Q_0)}^2 \\
  & = \| u \circ \kappabfm \|_{H^{1,0}(Q_0)}^2 
  + \big\| \Mcal (u\circ \kappabfm) \big\|_{H^{-1,0}(Q_0)}^2\\
  &\lesssim \| u \circ \kappabfm \|_{H^{1,0}(Q_0)}^2 \\
  &= \| u \|_{H^{1,0}(Q_T)}^2.
\end{align*}
\end{proof}

\begin{proof}[Proof of Theorem \ref{thm.positive_definiteness_calderon}]
We follow the proof of \cite[Theorem 3.11]{Costabel1990}. As 
above, we let the radius $R >0$ be big enough such that the ball 
$B_R$ contains the boundary $\Gamma_t$ for all $t$. We then 
write $\Omega_{t,R}^c = B_R \backslash \overline{\Omega}_t$ and 
$Q_{T,R}^c = \bigcup_{0< t< T}(\lbrace t \rbrace \times \Omega_{t,R}^c)$. 

Let $w \in H^{\frac12, \frac14} (\Sigma_T)$ and 
$\psi \in H^{-\frac12, -\frac14} (\Sigma_T)$. Apart from 
the boundary $\Sigma_T$, we define
\begin{equation}\label{eq.proof_positive_definite_def_u}
  u \isdef \widetilde{\Vcal} \psi - \widetilde{\Kcal} w. 
\end{equation}
From the jump relations (see Lemma \ref{lem.jump_relations}), we obtain
\begin{equation}\label{eq.proof_positive_definite_jump_rel}
  [\gamma_0 u] = - w, \quad [\gamma_1^- u] = - \psi.
\end{equation}
Using Definition \ref{def.layer_operators} immediately yields
\begin{equation}\label{eq.proof_positive_definite_rewriteA}
  \frac12 \left( \begin{bmatrix}
  \gamma_0 u|_{Q_T} \\ \gamma_1^- u|_{Q_T} 
  \end{bmatrix} + \begin{bmatrix}
  \gamma_0 u|_{Q_{T,R}^c} \\
  \gamma_1^- u|_{Q_{T,R}^c}
  \end{bmatrix} \right) = A \begin{bmatrix}
  \psi \\ w
  \end{bmatrix}.
\end{equation}
In view of \eqref{eq.proof_positive_definite_jump_rel} and 
\eqref{eq.proof_positive_definite_rewriteA}, we can rewrite 
the bilinear form as
\begin{equation}\label{eq.proof_positive_definite_rewrite_bilinear}
  \begin{aligned}
  \left\langle \begin{bmatrix}
  \psi \\w
  \end{bmatrix}, A \begin{bmatrix}
  \psi \\ w
  \end{bmatrix} \right\rangle & = \frac12 \left\langle \begin{bmatrix}
  \gamma_1^- u|_{Q_T} \\ \gamma_0 u|_{Q_T}
  \end{bmatrix} - \begin{bmatrix}
  \gamma_1^- u|_{Q_{T,R}^c} \\ \gamma_0 u|_{Q_{T,R}^c}
  \end{bmatrix}, \begin{bmatrix}
  \gamma_0 u|_{Q_T} \\ \gamma_1^- u|_{Q_T}
  \end{bmatrix} + \begin{bmatrix}
  \gamma_0 u|_{Q_{T,R}^c} \\ \gamma_1^- u_{Q_{T,R}^c}
  \end{bmatrix}\right\rangle \\
  & = \langle \gamma_1^- u|_{Q_T} , \gamma_0 u|_{Q_T} \rangle 
  - \langle \gamma_1^-u|_{Q_{T,R}^c}, \gamma_0 u|_{Q_{T,R}^c} \rangle,
\end{aligned}
\end{equation}
where we used \eqref{eq.def_of_duality_product_on_product_space}.

Since $u$ satisfies $(\partial_t - \Delta) u = 0$ in $Q_T$ and also in 
$Q_{T,R}^c$, we apply Green's first formula \eqref{eq.first_green_formula} 
(see Lemma \ref{lem.green_formula_tube}) to obtain	
\[
  \int_{Q_T} \| \nabla u \|^2 \, \drm (\xbfm, t) + d(u,u) 
  = \langle \gamma_1^- u |_{Q_T}, \gamma_0 u |_{Q_T} \rangle,
\]
while applying \eqref{eq.rewritten_green_formula} (see 
Lemma \ref{lem.green_formula_for_and_backward}) yields
\[
  \int_{Q_T} \| \nabla u \|^2 \, \drm(t,\xbfm) - d(u,u)  + \int_{\Omega_T} |u(T, \xbfm) |^2 \, \drm \xbfm 
  = \langle \gamma_1^- u |_{Q_T}, \gamma_0 u|_{Q_T} \rangle.
\]
Adding the two expressions together gives\footnote{At this point, it is crucial that we 
split the term $\langle \Vbfm, \nbfm \rangle$ in \eqref{eq.normaltrace} with the factor 
$\frac12$. If we choose the factor differently, say $\lambda$ and $1-\lambda$, we would 
obtain a boundary term involving $\int_0^T \int_{\Gamma_t} \langle \Vbfm, \nbfm \rangle 
(\gamma_0^{\on{int}} u)^2 \, \drm \sigma \drm t$, which would require an appropriate, 
not straight-forward treatment.}
\begin{equation}\label{eq.proof_positive_definite_traces_inside}
  \begin{aligned}
  \langle \gamma_1^- u|_{Q_T}, \gamma_0 u|_{Q_T} \rangle 
  &= \int_{Q_T} \| \nabla u \|^2 \, \drm (t,\xbfm) 
  + \frac12 \int_{\Omega_T} | u |^2 \, \drm \xbfm \\
  &\geq \int_{Q_T} \| \nabla u \|^2 \, \drm (t,\xbfm) .
  \end{aligned}
\end{equation}

On $Q_{T,R}^c$, we obtain analogously
\begin{equation}\label{eq.proof_positive_definite_traces_outside}
  \begin{aligned}
  - \langle \gamma_1^- u|_{Q_{T,R}^c}, \gamma_0 u |_{Q_{T,R}^c} \rangle 
  & = \int_{Q_{T,R}^c} \| \nabla u \|^2 \, \drm (\xbfm, t) 
\\
  & +\frac12 \int_{\Omega_{T,R}^c} | u(T,\xbfm) |^2 \, \drm \xbfm 
  - \int_0^T\int_{\partial B_R} u \partial_r u \, \drm \sigma \drm t,
  \end{aligned}
\end{equation}
where $\partial_r u$ denotes the normal derivative of $u$ at the boundary 
$(0,T)\times \partial B_R$. Inserting \eqref{eq.proof_positive_definite_traces_inside} 
and \eqref{eq.proof_positive_definite_traces_outside} into 
\eqref{eq.proof_positive_definite_rewrite_bilinear} yields
\begin{equation}\label{eq.proof_positive_definite_A_with_boundary_terms}
\begin{aligned}
  \left\langle \begin{bmatrix}
  \psi \\w
  \end{bmatrix}, A \begin{bmatrix}
  \psi \\ w
  \end{bmatrix} \right\rangle & \geq \int_{Q_T \cup Q_{T,R}^c} \| \nabla u \|^2 \, \drm (\xbfm, t) 
  - \int_0^T\int_{\partial B_R} u \partial_r u \, \drm \sigma \drm t.
  \end{aligned}
\end{equation}

According to \eqref{eq.proof_positive_definite_def_u}, $u|_{(0,T) \times \partial B_R}$ 
and $\partial_r u|_{(0,T) \times \partial B_R}$ are defined from $[w, \psi]^{\intercal}$ by 
the action of integral operators with smooth kernels. These integral operators as well 
as their adjoints are compact and therefore, using Young's inequality, there 
exists a compact operator $T_1 : \Hcal \to \Hcal$ such that
\[
  \left| \int_0^T\int_{\partial B_R} u \partial_r u \, \drm \sigma \drm t \right| \leq \left\langle \begin{bmatrix}
  \psi \\ w
  \end{bmatrix} , T_1 \begin{bmatrix}
  \psi \\ w
  \end{bmatrix} \right\rangle.
\]

According to \cite{Costabel1990}, $H^{1, \frac12}_{;0,} (Q_0)$ embeds 
compactly into $L^2 (Q_0)$ and, therefore, also $H^{1, \frac12}_{;0,} (Q_T)$ 
embeds compactly into $L^2 (Q_T)$ due to the smooth mapping $\kappabfm$.
Then, in view of the mapping properties of $\widetilde{\Vcal}$ and $\widetilde{\Kcal}$ 
(see Lemma \ref{lem.mapping_properties_single_layer_pot} and Lemma 
\ref{lem.mapping_properties_double_layer_pot}), we obtain the existence of 
another compact operator $T_2 \colon \Hcal \to \Hcal$, such that, due to 
the definition of the norm, we have
\begin{align*}
  \int_{Q_T \cup Q_{T,R}^c}\!\! \|\nabla u \|^2 \, \drm (t,\xbfm) 
  &= \left\| u|_{Q_T} \right\|^2_{H^{1,0}(Q_T)} + \big\| u|_{Q_{T,R}^c} \big\|^2_{H^{1,0} (Q_{T,R}^c)}  
  - \int_{Q_T \cup Q_{T,R}^c}\!\! | u |^2 \, \drm(t,\xbfm) \\
  &=  \left\| u|_{Q_T} \right\|^2_{H^{1,0}(Q_T)} + \big\| u|_{Q_{T,R}^c} \big\|^2_{H^{1,0} (Q_{T,R}^c)}   
  - \left\langle \begin{bmatrix} \psi \\ w \end{bmatrix}, T_2 \begin{bmatrix} \psi \\ w \end{bmatrix} \right\rangle.
\end{align*}
Due to Lemma \ref{lem.norm_equivalence_H_V}, the norms 
of $u|_{Q_T}$ in $H^{1,0}(Q_T)$ and $H^{1,\frac12}(Q_T)$ are 
equivalent, and likewise those of $u|_{Q_{T,R}^c}$. Hence, 
\eqref{eq.proof_positive_definite_A_with_boundary_terms} induces
\begin{equation}\label{eq.proof_positive_definite_operator_below}
\begin{aligned}
  &\left\langle \begin{bmatrix} \psi \\w \end{bmatrix}, 
  A \begin{bmatrix} \psi \\ w \end{bmatrix} \right\rangle 		 
  \geq  \alpha \left( \left\| u|_{Q_T}\right\|^2_{H^{1,\frac12}(Q_T)} 
  + \big\| u|_{Q_{T,R}^c}\big\|^2_{H^{1,\frac12} (Q_{T,R}^c )}  \right) \\
  &\hspace*{4cm} - \left\langle \begin{bmatrix} \psi \\ w \end{bmatrix}, 
  (T_1 + T_2 ) \begin{bmatrix} \psi \\ w \end{bmatrix} \right\rangle.
  \end{aligned}
\end{equation}

Using the jump relations \eqref{eq.proof_positive_definite_jump_rel} 
and then the trace lemmata (see Subsection \ref{subsec.trace_op_non_cylindrical_domains} 
and Lemma \ref{lem.mapping_properties_gamma_1_minus}) yields
\begin{align*}
  \| w \|_{H^{\frac12, \frac14} (\Sigma_T)} 
  &= \big\| \gamma_0 u|_{Q_T} - \gamma_0 u |_{Q_{T,R}^c}\big\|_{H^{\frac12, \frac14} (\Sigma_T)} \\
  & \lesssim \left(\| u |_{Q_T}\|_{H^{1, \frac12}(Q_T)} 
  + \big\| u|_{Q_{T,R}^c}\big\|_{H^{1, \frac12} (Q_{T,R}^c)} \right)
\end{align*}
and similarly
\begin{align*}
  \| \psi \|_{H^{-\frac12, -\frac14} (\Sigma_T)} 
  &= \big\| \gamma_1^- u|_{Q_T} - \gamma_1^- u |_{Q_{T,R}^c}\big\|_{H^{-\frac12, -\frac14} (\Sigma_T)} \\
  & \lesssim \left(\big\| u |_{Q_T}\big\|_{H^{1, \frac12}(Q_T)} 
  + \big\| u|_{Q_{T,R}^c}\big\|_{H^{1, \frac12} (Q_{T,R}^c)} \right).
\end{align*}
Looking at \eqref{eq.proof_positive_definite_operator_below} 
we thus have the existence of a constant $\alpha > 0$ with
\begin{align*}
  \left\langle \begin{bmatrix} \psi \\w \end{bmatrix}, 
  (A + T_1 + T_2) \begin{bmatrix} \psi \\ w \end{bmatrix} \right\rangle
  \geq \alpha \left( \| \psi \|^2_{H^{-\frac12, -\frac14} (\Sigma_T)} 
  + \| w \|^2_{H^{\frac12, \frac14} (\Sigma_T)} \right),
\end{align*}
which is the first assumption in Lemma \ref{lem.proof_positive_definite_help_lemma}.
	
It remains to prove the positivity assumption in Lemma 
\ref{lem.proof_positive_definite_help_lemma}. To that end, we show 
that the term $\int_0^T\int_{\partial B_R} u \partial_r u \, \drm \sigma \drm t$ 
in \eqref{eq.proof_positive_definite_A_with_boundary_terms} goes to zero as 
$R \to \infty$. Let $0 < R_0 < R$ such that $\overline{\Omega}_t \subset B_{R_0}$ 
for all $t$ and set $Q_{T, R_0}^c = (0,T) \times (B_{R_0} \backslash \overline{\Omega}_0)$. 
We can use Green's second formula from Lemma \ref{lem.green_formula_tube} 
for $v = \widetilde{G} (T-t,  \xbfm)$ with $(t_0, \xbfm_0) \notin Q_{T, R_0}^c$. 
Thus, outside of $Q_{T, R_0}^c$ for $\| \xbfm \| > R_0$, the function $u$ 
coincides with 
\[ 
u_0 \isdef \widetilde{\Vcal} \psi_0 - \widetilde{\Kcal} w_0,
\]
where the potentials take the densities on $\Sigma_{R_0} 
\isdef (0,T) \times \partial B_{R_0}$ given by
\[ 
w_0 \isdef u|_{\Sigma_{R_0}}, \quad \psi_0 \isdef \partial_r u |_{\Sigma_{R_0}}.
\]

Since the singularity of $u$ lies on the boundary $\Sigma_T$, the 
densities $w_0$ and $\psi_0$ are smooth and also the boundary 
$\Sigma_{R_0}$ is smooth. Therefore, we can estimate
$u|_{\Sigma_R} = u_0 |_{\Sigma_R}$ and $\partial_r u|_{\Sigma_R} 
= \partial_r u_0 |_{\Sigma_R}$ for $R > R_0$ by looking at the 
behaviour of the fundamental solution $G$. Because the fundamental 
solution is the same for the cylindrical and the non-cylindrical case, 
we estimate
\[ 
  | G(t,  \xbfm) | \leq C_{\mu} t^{-\mu} \| \xbfm \|^{2 \mu - d} 
  	\quad \text{for all } \mu \in \Rbbb
\]
and we obtain a similar estimate for $\nabla G$. Then, for 
finite $T$, we have
\[ 
u = \Ocal (R^{-d}), \quad \partial_r u = \Ocal (R^{-d-1}) 
\quad \text{as } \| \xbfm \| = R \to \infty.
\]
Therefore, since the integrand is of order $\Ocal (R^{-d-d-1})$ 
and the measure of the boundary $\partial B_R$ is of order 
$\Ocal (R^{d-1})$, we obtain
\[ 
\int_0^T\int_{\partial B_R} u \partial_r u \, \drm \sigma \drm t 
= \Ocal (R^{-d-2} ) \to 0 \quad\text{as } \| \xbfm \| = R \to \infty. 
\]

Since the left hand side in \eqref{eq.proof_positive_definite_A_with_boundary_terms} 
is independent of $R$, we can conclude that 
\[ 
\lim_{R \to \infty} \int_{Q_{T,R}^c} \| \nabla u \| ^2 \, \drm (t,\xbfm)
\]
is finite and
\[ 
  \left\langle \begin{bmatrix}\psi \\ w \end{bmatrix}, 
  A \begin{bmatrix} \psi \\ w \end{bmatrix} \right\rangle 
  \geq \int_{(0,T) \times ( \Rbbb^d \backslash \Gamma_t )} \| \nabla u \|^2 \, \drm (t,\xbfm). 
\]

Assume that the right hand side vanishes. Then, since $u$ is smooth 
enough, we obtain that $u(t, \cdot)$ is constant on $\Omega_t$ and 
$\Rbbb^d \backslash \overline{\Omega}_t$ for every $t \in (0,T)$. Since $u = 0$ 
on $\Omega_0$, we thus obtain that $u \equiv 0$ on $(0,T) \times \Rbbb^d$. 
From the jump relations \eqref{eq.proof_positive_definite_jump_rel}, we 
obtain $w = 0$ and $\psi = 0$. This implies the positivity assumption 
\eqref{eq.def_of_duality_product_on_product_space} of Lemma 
\ref{lem.proof_positive_definite_help_lemma} and the claim in 
the theorem follows immediately.
\end{proof}

Having the main result Theorem \ref{thm.positive_definiteness_calderon} 
at hand, we can state a few corollaries along the lines of \cite[Corollary 3.13, 
Corollary 3.14, Remark 3.15, Corollary 3.16, Corollary 3.17]{Costabel1990}.

\begin{corollary}
The single layer operator
\[ 
  \Vcal \colon H^{-\frac12, -\frac14} (\Sigma_T) \to H^{\frac12, \frac14} (\Sigma_T) 
\]
is an isomorphism and there exists $\alpha > 0$ such that
\begin{equation}\label{eq.ellipticity_single_layer_op}
  \langle \Vcal \psi, \psi \rangle  \geq \alpha \| \psi \|^2_{H^{-\frac12, -\frac14} (\Sigma_T)} 
  	\quad \text{for all } \psi \in H^{-\frac12, -\frac14} (\Sigma_T).
\end{equation}	
The hypersingular operator 
\[ 
  \Dcal \colon H^{\frac12, \frac14} (\Sigma_T) \to H^{-\frac12, -\frac14} (\Sigma_T)
\]
is an isomorphism and there exists $\alpha > 0$ such that
\begin{equation}\label{eq.ellipticity_hypersingular_op}
  \langle \Dcal w,  w \rangle \geq \alpha \| w \|^2_{H^{\frac12, \frac14} (\Sigma_T)} 
  	\quad \text{for all } w \in H^{\frac12, \frac14} (\Sigma_T).
\end{equation}
\end{corollary}

\begin{proof}
As in the proof of \cite[Corollary 3.13]{Costabel1990}, the coercivity estimates 
\eqref{eq.ellipticity_single_layer_op} and \eqref{eq.ellipticity_hypersingular_op} 
result from Theorem \ref{thm.positive_definiteness_calderon} by using the 
special cases $w = 0$ and $\psi = 0$, respectively. In view of the continuity 
of $\Vcal$ and $\Dcal$, this leads to the invertibility of the operators.
\end{proof}

\begin{corollary}
The operators
\begin{align*}
  \frac12 \on{id} + \Kcal, \ \ \frac12 \on{id} - \Kcal 
  	&\colon H^{\frac12, \frac14} (\Sigma_T) \to H^{\frac12, \frac14} (\Sigma_T), \\
  \frac12 \on{id} + \Kcal ', \ \ \frac12 \on{id} - \Kcal ' 
  	&\colon H^{-\frac12, -\frac14} (\Sigma_T) \to H^{-\frac12, -\frac14} (\Sigma_T)
\end{align*}
are isomorphisms.
\end{corollary}

\begin{proof}
We again follow the proof of \cite[Corollary 3.14]{Costabel1990} directly. 
From the projection property \eqref{eq.projection_property_calderon}, more 
specifically from \eqref{eq.projection_property_calderon_rewritten}, we obtain
\begin{align*}
  \left( \frac12 \on{id} + \Kcal \right) \left( \frac12 \on{id} - \Kcal \right) & = \Vcal \Dcal, \\
  \left( \frac12 \on{id} + \Kcal ' \right) \left( \frac12 \on{id} - \Kcal ' \right) & = \Wcal \Vcal.
\end{align*}
Since the right hand sides are isomorphisms, we immediately arrive 
at the claim.
\end{proof}

\begin{remark}
The other two relations gained from \eqref{eq.projection_property_calderon_rewritten} 
lead to
\[ 
  \Vcal^{-1} \Kcal \Vcal = \Kcal ' = \Dcal \Kcal \Dcal^{-1}. 
\]
\end{remark}

\begin{corollary}
The unique solution $u \in H^{1, \frac12}_{;0,} (Q_T)$ of the Dirichlet problem
\[
  \begin{aligned}
  (\partial_t - \Delta) u & = 0 \ \ && \text{in } Q_T, \\
  \gamma_0 u & = g \ \ && \text{on } \Sigma_T,
  \end{aligned}
\]
with $g \in H^{\frac12, \frac14}(\Sigma_T)$ can be represented 
\begin{enumerate}[(i)]
\item as $u = \widetilde{\Vcal} \psi - \widetilde{\Kcal} g$, where 
$\psi \in H^{-\frac12, -\frac14} (\Sigma_T)$ is the unique solution 
of the first kind integral equation
\[ 
\Vcal \psi = \left(\frac12 \on{id} + \Kcal \right) g. 
\]
\item as $u = \widetilde{\Vcal} \psi - \widetilde{\Kcal} g$, where 
$\psi \in H^{-\frac12, -\frac14} (\Sigma_T)$ is the unique solution 
of the second kind integral equation
\[ 
\left( \frac12 \on{id} - \Kcal ' \right) \psi = \Dcal g.
\]
\item as $u = \widetilde{\Vcal} \psi$, where $\psi \in 
H^{-\frac12, -\frac14}(\Sigma_T)$ is the unique solution of 
the first kind integral equation 
\[ 
\Vcal \psi = g.
\]
\item as $u = \widetilde{\Kcal} w$, where $w \in 
H^{\frac12, \frac14}(\Sigma_T)$ is the unique solution of the 
second kind integral equation
\[ 
\left( \frac12 \on{id} - \Kcal \right) w = - g. 
\]
\end{enumerate}
In (i) and (ii), it particularly holds $\psi = \gamma_1^- u$ on $\Sigma_T$. 
\end{corollary}

\begin{proof}
We can again use directly the idea of the proof of 
\cite[Corollary 3.16]{Costabel1990}, which are the uniqueness 
results from above and the jump relations given in Lemma 
\ref{lem.jump_relations_one_sided}.
\end{proof}

\begin{corollary}\label{cor.ex_and_unique_Neumann}
The unique solution $u \in H^{1, \frac12}_{;0,}(Q_T)$ of the Neumann problem
\[
  \begin{aligned}
  (\partial_t - \Delta) u & = 0 \ \ && \text{in } Q_T, \\
  \gamma_1^- u & = h \ \ && \text{on } \Sigma_T,
  \end{aligned}
\]
with $h \in H^{-\frac12, -\frac14} (\Sigma_T)$ can be represented
\begin{enumerate}[(i)]
\item as $u = \widetilde{\Vcal} h - \widetilde{\Kcal} w$, where 
$w \in H^{\frac12, \frac14}(\Sigma_T)$ is the unique solution of 
the second kind integral equation
\[ 
\left( \frac12 \on{id} + \Kcal \right) w = \Vcal h.
\]
\item as $u = \widetilde{\Vcal} h - \widetilde{\Kcal} w$, where 
$w \in H^{\frac12, \frac14} (\Sigma_T)$ is the unique solution 
of the first kind integral equation
\[ 
\Dcal w = \left( \frac12 \on{id} - \Kcal' \right) h.
\]
\item as $u = \widetilde{\Vcal} \psi$, where $\psi \in 
H^{-\frac12, -\frac14} (\Sigma_T)$ is the unique solution 
of the second kind integral equation
\[ 
\left(\frac12 \on{id} + \Kcal ' \right) \psi = h.
\]
\item as $u = \widetilde{\Kcal} w$, where 
$w \in H^{\frac12, \frac14} (\Sigma_T)$ is the unique 
solution of the first kind integral equation
\[ 
\Dcal w = - h.
\]
\end{enumerate}
In (i) and (ii), we have that $w = \gamma_0 u$ on $\Sigma_T$. 
\end{corollary}

\begin{proof}
The proof follows by the same arguments as in the proof of 
\cite[Corollary 3.17]{Costabel1990}, which is similar to the 
respective proof for the Dirichlet problem.
\end{proof}

\section{Conclusion}\label{sec.conclusion}
In this article, we considered the heat equation on a time-varying 
(so-called non-cylindrical) domain. In contrast to the problem on a 
cylindrical domain, we used a modified Neumann trace operator 
containing a term which is dependent on the velocity of the moving 
surface. We were able to show the mapping properties of the layer 
operators by following the proofs of Costabel \cite{Costabel1990}. 
To this end, we heavily used the fact that the non-cylindrical domain 
is a mapped cylindrical domain. Then, using mapped anisotropic 
Sobolev spaces, we obtain analogous mapping properties and are 
also able to prove existence and uniqueness of solutions of the 
Dirichlet and of the Neumann problem.

\subsection*{Acknowledgement}
This reseach is in part supported by the National Science
Foundation under grant DMS-1720431.
\bibliographystyle{plain}

\end{document}